 \documentclass[preprint,11pt]{elsarticle}

\usepackage{amssymb}
\usepackage[export]{adjustbox}
\usepackage{amsfonts}
\usepackage{amsmath,xfrac}
\usepackage{mathtools}   
\usepackage{empheq}
\usepackage{stackrel}

\usepackage{graphicx,natbib,amsthm,amsmath,tools,eepic,epic}
\usepackage{float,subfig}
\usepackage{xcolor}
\usepackage{tikz}
\usepackage{lineno}
\usetikzlibrary{patterns}
\usepackage{pgfplots}
\usepackage{multirow}
\usetikzlibrary{shadows}
\usetikzlibrary{plotmarks}
\usepackage[margin=3.0cm]{geometry}
\usepackage{hyperref}
\hypersetup{
colorlinks=true,
linkcolor=blue,
filecolor=magenta,
urlcolor=cyan,
citecolor=gray,
}

\usepackage{lineno}
\usepackage{lipsum}

\newtheorem{lemma}{Lemma}[section]
\newtheorem{proposition}{Proposition}[section]
\newtheorem{definition}{Definition}[section]
\newtheorem{remark}{Remark}[section]

\definecolor{OliveGreen}{HTML}{556B2F}
\definecolor{prettyBlue}{HTML}{2196F3}
\definecolor{MidnightBlue}{HTML}{191970}
\definecolor{gray}{rgb}{0.5, 0.5, 0.5}
\definecolor{BrickRed}{rgb}{0.8, 0.25, 0.33}
\definecolor{Sepia}{rgb}{0.44, 0.26, 0.08}
\definecolor{NavyBlue}{rgb}{0.0, 0.0, 0.5}
\definecolor{ForestGreen}{rgb}{0.0, 0.27, 0.13}
\definecolor{Sepia}{rgb}{0.44, 0.26, 0.08}

\definecolor{deepblue}{rgb}{0,0,0.5}
\definecolor{deepred}{rgb}{0.6,0,0}
\definecolor{deepgreen}{rgb}{0,0.5,0}
\definecolor{carmine}{rgb}{0.59, 0.0, 0.09}
\definecolor{dartmouthgreen}{rgb}{0.05, 0.5, 0.06}
\definecolor{tawny}{rgb}{0.8, 0.34, 0.0}
\definecolor{tealblue}{rgb}{0.21, 0.46, 0.53}
\definecolor{cordovan}{rgb}{0.54, 0.25, 0.27}
\definecolor{oxfordblue}{rgb}{0.0, 0.13, 0.28}

\journal{Computer Physics Communications}

\begin{document}

\begin{frontmatter}



\newdefinition{asp}{Assumption}
\newdefinition{rmk}{Remark}

\title{Families of hybridizable interior penalty discontinuous Galerkin methods for degenerate advection-diffusion-reaction problems}

\author[add1]{Gr\'egory Etangsale}
\ead{gregory.etangsale@univ-reunion}
\author[add2]{Marwan Fahs}
\ead{fahs@unistra.fr}
\author[add1]{Vincent Fontaine\corref{cor1}}
\ead{vincent.fontaine@univ-reunion}
\cortext[cor1]{Corresponding and Principal author}
\author[add3]{Ali Raeisi Isa-Abadi }
\ead{ali_raeisi@sku.ac.ir}

\address[add1]{Department of Building and Environmental Sciences, University of La R\'eunion - South Campus, France}
\address[add2]{$Universit\acute{e}$ de Strasbourg, CNRS, ENGEES, LHYGES UMR 7517, F-67000 Strasbourg, France}
\address[add3]{Department of Water Engineering, Shahrekord University, Shahrekord, Iran}


\begin{abstract}
We analyze families of primal high-order hybridizable discontinuous Galerkin (HDG) methods for solving degenerate (second-order) elliptic problems. One major trouble regarding this class of PDEs concerns its mathematical nature, which may be nonuniform over the domain. 
Due to the local degeneracy of the diffusion term, it can be purely hyperbolic in a subregion and elliptic in the rest. This problem is thus quite delicate to solve since the exact solution is discontinuous at interfaces separating both elliptic and hyperbolic parts. The proposed HDG method is developed in a unified and compact fashion. It can efficiently handle pure diffusive or advective regimes and intermediate regimes that combine the above mechanisms for a wide range of P\'eclet numbers, including the delicate situation of local evanescent diffusion. To this end, an adaptive stabilization strategy based on addition of jump-penalty terms is then considered. A $\theta$-upwind-based scheme is favored for the hyperbolic region, and an inspired Scharfetter--Gummel-based technique is preferred for the elliptic region. The well-posedness of the HDG method is also discussed by analyzing the consistency and discrete coercivity properties. Extensive numerical experiments are finally considered to verify the model's robustness for all the abovementioned regimes. 
\end{abstract}

\begin{keyword}
Primal hybridizable discontinuous Galerkin\sep interior penalty methods\sep degenerate second-order elliptic problems\sep adaptive penalty strategy\sep upwind-based scheme \sep Scharfetter--Gummel scheme
\sep extensive numerical experiments


\end{keyword}

\end{frontmatter}


\section{Introduction}
\label{Intro}

Degenerate second-order elliptic equations are well-established models to describe a wide variety of phenomena in real-life applications, such as pure diffusion or advection problems, and mixed problems combining the above mechanisms for a wide range of P\'eclet numbers \citep{oleinik1973}. A detrimental situation may arise in the context of locally evanescent diffusivity. Indeed, its mathematical nature is nonuniform over the entire domain, as it can be purely hyperbolic in a subregion and elliptic in the rest. Consequently, the state variable can be discontinuous at interfaces separating both subregions according to the wind flow sense \citep{dipietro:hal-01079342,dipietro:hal-01023302,ESZIMA08}. This critical situation is easily encountered in the context of mass transport in fractured porous media. It is well known that fractures deeply affect the transport phenomena since they represent the preferential fluid flow paths. The large variability of the velocity's magnitude indicates that the advection mechanism is predominant in the fractures as compared with the rest of the domain. During recent decades, different authors have analyzed this model problem, although mainly in the context of discontinuous Galerkin (DG) methods (see, e.g., \citep{dipietro:hal-01023302,ESZIMA08,HSSSIAM02} and the extensive references therein). The success story of DG methods is because they combine advantages of finite volume and finite element methods, and they are well-suited to capture large gradients or discontinuities of exact solutions \citep{RFGYF2020104744}. Despite all of these advantages, DG methods are generally more expensive than most other numerical methods due to their high number of coupled degrees of freedom (DOFs) and their large stencils. It is in this context that the HDG methods were initially devised: as a way to circumvent these drawbacks.\par 
The HDG methods were first introduced by Cockburn \al in \citep{CGLSIAM09} and have been applied successfully to various physical problems \citep{Cockburn2016,CQS2012, DijouxA2M19,Nguyen2011NS,NPClinearconv,LehrenfeldTH,Oikawa2015,KHCRSIAM2019}. They can be considered as a new class of DG methods that are eligible for static condensation. For that application, an additional discrete variable is introduced corresponding to the trace approximation of the state variable on the mesh skeleton. Thus, the interior-based DOFs can be easily eliminated by solving a local problem at the element level so that only skeleton-based DOFs remain. The problem is then closed by (weakly) imposing transmission conditions on the mesh skeleton, leading to a smaller and sparser final matrix system. In practice, DG methods and their HDG counterparts do not coincide because the latter use a richer definition of numerical traces. Consequently, HDG methods turn out to be more accurate than their predecessors in many situations, and they are thus more efficiently implementable and highly parallelizable \citep{etangsale:hal-02893064,Kirby2012}. Despite all of these assets, the literature is relatively scarce concerning the resolution of degenerate elliptic equations by the class of HDG methods, which is the purpose of the present work. To the best of our knowledge, only Di Pietro \al recently designed a primal discontinuous skeletal method based on the hybrid-high order (HHO) formalism for the diffusive part to address this kind of issue \citep{dipietro:hal-01079342,dipietro:hal-01818201}.\par
In the present paper, we focus instead on the class of interior penalty HDG methods denoted by H-IP and its three well-known variants, namely, the incomplete (H-IIP), symmetric (H-SIP), and nonsymmetric (H-NIP) schemes \citep{FabienNM2019,wells2011analysis}. Indeed, the families of interior penalty methods are well-suited for solving degenerate elliptic equations since the diffusion term might not be invertible at every point of the domain; \ie they belong to the class of $\bkappa$-methods. For this aim, both diffusive and advective-reactive contributions are discretized separately. The stability of these contributions is ensured by adding jump-penalty terms, which correspond to the discrepancy between interior- and interface-based DOFs, on the mesh skeleton. The stabilization penalty parameters are selected in accordance with the nature of the local cellwise problem reducing to an upwinding-based scheme in the hyperbolic subregion \citep{BMS2004M3AS} and the Scharfetter--Gummel (SG) scheme elsewhere \citep{BREZZI2006681,dipietro:hal-01079342}. Thus, the stated H-IP formalism can treat in an automatic fashion the pure diffusion or  advection-reaction processes or (mixed) advection-diffusion-reaction processes characterized by a diffusion- or advection-dominated regime, \ie a wide range of P\'eclet numbers - including the delicate situation of local evanescent diffusion. A stability analysis is then investigated by establishing the consistency and discrete coercivity properties. 
Numerical experiments are also presented to prove the following assertions, such as the high-order accuracy and robustness of the discretization method.\par
The material is organized as follows. In Section \ref{BVP}, we describe the homogeneous Dirichlet boundary value problem in the sense of Fichera \citep{oleinik1973} by introducing some specific notations, and we precisely define the corresponding discrete setting in Section \ref{Discrete_setting}. In Section \ref{HIP}, we derive the discrete bilinear and linear operators of the discretization method, briefly discuss the static condensation and stability analysis, and precisely delineate the stabilization strategy. In Section \ref{Numerical_exp}, extensive numerical experiments are investigated using $h$- and $p$-refinement strategies for all abovementioned regimes. We end with some concluding remarks and perspectives.

\section{Boundary value problem}
\label{BVP}


We consider the stationary linear advection-diffusion-reaction model problem in its conservative form, 
\begin{equation}
\label{primal-ADR}
\begin{array}{rccll}
\quad\dvg(-\bkappa\grd u+\bbeta u)+\gamma u&=&f &\mbox{ in }& \Omega,
\end{array}
\end{equation}
where $\Omega\subset\mathbb{R}^d$ is a bounded polyhedral domain ($d\geq 2$) with boundary $\partial\Omega\subset\R^{d-1}$. The boldface fonts are used throughout the paper to characterize any vector- or matrix-valued functions. In the physical context of contaminant transport in a porous media, $\bkappa:\Domain\rightarrow\R^{d,d}$ represents an anisotropic heterogeneous dispersion tensor, which is itself a function of the Darcy velocity field $\bbeta:\Domain\rightarrow\R^{d}$, $\gamma:\Domain\rightarrow\R^{+}$, the reaction coefficient, and $f:\Domain\rightarrow\R$, a forcing term. We then assume that the constitutive coefficients of \eqref{primal-ADR} satisfy the following minimal regularity requirements:
\begin{itemize}
\item[$\ast$]$\bkappa\in[L^{\infty}(\Omega)]^{d,d}$ is a symmetric positive semidefinite matrix-valued function verifying that
\begin{equation}
\label{spsd}
\underbar{$\kappa$}\norm{\vect{\zeta}}{}{}^2\leq\vect{\zeta}^t\bkappa(x)\vect{\zeta}\leq\bar{\kappa}\norm{\vect{\zeta}}{}{}^2,\quad\forall\vect{\zeta}\in\R^d,\quad\ale\quad x\in\Domain.
\end{equation}
where $\bar{\kappa}\geq\underbar{$\kappa$}\geq 0$ denote the largest and smallest eigenvalues of $\bkappa$, respectively. Under the hypothesis \eqref{spsd}, we shall assume the existence of a subdomain $\Domaine$ (resp. $\Domainh$) corresponding to the elliptic (resp. hyperbolic) region such that 
\begin{equation}
\label{hypell}
\begin{array}{lclc}
\Domaine &\eqbydef &\{x\in\Domain\,:\;\vect{\zeta}^t\bkappa(x)\vect{\zeta}> 0,\quad\forall\vect{\zeta}\in\R^d\},\\
\Domainh &\eqbydef &\{x\in\Domain\,:\;\vect{\zeta}^t\bkappa(x)\vect{\zeta}=0,\quad\forall\vect{\zeta}\in\R^d\},
\end{array}
\end{equation}
verifying that $\Domain=\Domaine\cup\Domainh$ and $\Domaine\cap\Domainh=\emptyset$ (nonoverlapping subregions).
\item[$\ast$]$\bbeta\in[L^{\infty}(\Omega)]^{d}$ is \st $\dvg\bbeta\in L^{\infty}(\Omega)$, 
\item[$\ast$]$\gamma\in L^{\infty}(\Omega)$ verifying that the following standard coercivity condition holds 
\begin{equation}
\label{coercivity-cond}
\exists\gamma_0\in\R\quad\textrm{s.t.}\quad\gamma(x)+\dfrac{1}{2}\dvg\bbeta(x)\geq\gamma_0>0, \quad\ale\quad x\in\Domain.
\end{equation}
\end{itemize}
Let us now introduce the following disjoint boundary sets as defined by Ole\u{\i}nik and Radkevi\v{c} \citep{oleinik1973} (see, \eg Di Pietro \al \citep{dipietro:hal-01079342} and references therein);
\begin{equation}
\label{Fichera}
\begin{array}{clc}
&\gammam\eqbydef\{x\in\BOmega\,:\;\vect{n}^t\bkappa(x)\vect{n}>0\quad\textrm{or}&\bbeta\cdot\vect{n}<0\},\\
&\gammap\eqbydef\{x\in\BOmega\,:\;\vect{n}^t\bkappa(x)\vect{n}=0\quad\textrm{and} &\bbeta\cdot\vect{n}\geq 0\},
\end{array}
\end{equation}
where $\vect{n}$ denotes the unit outward normal to $\BOmega$. Thus, the disjoint subsets $\gammapm$ will be referred to as the \textit{nondegenerate inflow} and \textit{degenerate outflow/no-flow} parts of the boundary $\BOmega$, respectively, verifying that $\BOmega=\gammam\cup\gammap$ and $\gammam\cup\gammap=\emptyset$. For clarity of our exposition, we supplement the partial differential equation \citeq{primal-ADR} with the following homogeneous Dirichlet boundary conditions, namely, 
\begin{equation}
\label{DirBC}
u=0\quad\mathrm{on}\quad\gammam.
\end{equation}
Let us remark that the Dirichlet boundary conditions are only prescribed for portions of $\BOmega$ touching the elliptic region or the hyperbolic region, provided that the advective field flows into the domain. Let us now summarize some physical situations commonly encountered in the literature for which Dirichlet boundary conditions \eqref{DirBC} adapt automatically:
\begin{itemize}
\item[$\blacktriangleright$] \underline{Nondegenerate problems}. Here, $\bkappa$ is assumed to be a symmetric positive definite matrix-valued function on the whole domain $\Domain$. Hence, $\Domaine=\Domain$ and the Dirichlet boundary condition is automatically enforced for the whole boundary $\gammam=\BOmega$. This situation includes the pure diffusive regime as well as the mixed advective-diffusive regime. Here, the weak solution $u$ of the problem \eqref{primal-ADR} respects some regularity requirements \ie $u\in H^1(\Domain)$, but it may present some sharp fronts along the characteristic direction $\bbeta$, particularly for the advective-dominated regime.  
\item[$\blacktriangleright$] \underline{Fully degenerate problems}. Here, $\Domaine=\emptyset$ and we recover the standard advection-reaction problem defined in $\Domainh$ respecting the coercivity condition \eqref{coercivity-cond} and the usual definition of the inflow $\Gamma^{-}$ or outflow $\Gamma^{+}$ parts of the boundary $\BOmega$. Following these assumptions, the problem is well-posed with no smoothing properties \ie $u\in\LO$, and discontinuities in the solution $u$ induced by $f$ will propagate along the flow field $\bbeta$, giving rise to internal layers. 
\item[$\blacktriangleright$] \underline{Locally degenerate problems}. We assume that both $\Domainh$ and $\Domaine$ are nonempty subsets. The model problem is then purely hyperbolic in $\Domainh$ and elliptic in the rest. Thus, we now define the common interface $\interf\eqbydef\{x\in\Domain\,:\;\partial\Domainh\cap\partial\Domaine\}$. We emphasize that such problems are particularly delicate to solve since the solution can be discontinuous at the portion $\interfm\eqbydef\{x\in\interf\,:\; \bbeta(x)\cdot\vect{n}_{\interf}<0\}$, where $\vect{n}_{\interf}$ is an (arbitrary) oriented unit normal vector pointing out of the elliptic region. Concretely, $\interfm$ corresponds to a subset of $\interf$ where the advection field flows from the hyperbolic side to the elliptic side. Thus, we set $\interfp\eqbydef\interf\backslash\interfm$. 
\end{itemize}
For a given forcing term $f\in L^{2}(\Domain)$, the continuous problem reads:
\begin{equation}
\label{continuous-deg-problem}
\begin{array}{rccll}
\quad\dvg(-\bkappa\grd u+\bbeta u)+\gamma u&=&f &\mbox{ in }& \Omega,\\
\jump{u}&=&\vect{0}&\mbox{ on }& \interfp,\\
u&=&0 &\mbox{ on }&\gammam,
\end{array}
\end{equation}
where $\jump{\cdot}$ denotes the standard DG-jump trace operator as defined in \cite{ABCMUnified}. The well-posedness of the boundary value problem \eqref{continuous-deg-problem} has been analyzed by Ole\u{\i}nik and Radkevi\v{c} in \citep{oleinik1973}. They proved the existence and uniqueness of a weak solution for homogeneous and nonhomogeneous Dirichlet boundary conditions, respectively. The main objective of the present paper is to propose, in a unified formalism, an inspired interior penalty HDG method that can treat all of these abovementioned physical situations in an automated fashion. Before doing so, let us specify the discrete setting concerning mesh assumptions, the definition of trace operators and the approximation spaces that will be used later in the rest of this paper.

\section{Discrete setting}
\label{Discrete_setting}

\subsection{Mesh notations}

Let $h$ be a positive parameter, and assume without loss of generality that $h\leq 1$. We denote by $\Th{ell}$ (resp. $\Th{hyp}$) a conformal partition of the subdomain $\Domaine$ (resp. $\Domainh$) satisfying $\Th{ell}\cap\Th{hyp}=\emptyset$. The set of all mesh elements is denoted by $\Th{}$, i.e., $\Th{}\eqbydef\Th{ell}\cup\Th{hyp}$, where $h$ stands for the largest diameter of all elements. We precisely state here that the sets of mesh elements $\Th{ell}$ and $\Th{hyp}$ can be composed of several types of geometric elements, \ie hybrid meshes. Following our notation, the generic term \textit{interface} indicates a $(d-1)$-dimensional geometric object with a positive measure, i.e., an edge if $d=2$ and a face if $d=3$. The set of boundary interfaces is denoted by $\Fhb$, i.e., $\f{}\in\Fhb$ if there exists $\e{}$ in $\Th{}$ such that $\f{}\eqbydef\bnd{\e{}}\cap\BOmega$. We assume that the set $\Fhb$ coincides with the disjoint boundary partition $\Partedgbpm$, i.e., the boundary interfaces lying entirely in one of the subsets $\Gamma^{\pm}$. Likewise, we denote by $\Fhi$ the set of interior interfaces, i.e., $\f{}\in\Fhi$ if there exists $\e{1}$ and $\e{2}$ in $\Th{}$ such that $\f{}\eqbydef\bnd{\e{1}}\cap\bnd{\e{2}}$. The set of all interfaces is denoted by $\Partedg$, i.e., $\Fh{}\eqbydef\Fhi\cup\Fhb$, and we set $\Fh{\pm}\eqbydef\Fhi\cup\Partedgbpm$. In particular, we denote by $\Ih{}$ the subset of $\Fhi$, which belongs to $\interf$ ($\dinterf\subset\Fhi$), i.e., $\dinterf$ corresponds here to the discrete counterpart of $\interf$. We assume first that for any interface $\f{}\in\Ih{}$ (i) there exist $\e{1}\in\Th{ell}$ and $\e{2}\in\Th{hyp}$ such that $\f{}\eqbydef\bnd{\e{1}}\cap\bnd{\e{2}}$, and (ii) $\f{}$ lies entirely in one of the disjoint subsets $\Ih{\pm}$ corresponding to the discrete counterpart of $\interfpm$, respectively. Moreover, for any mesh element $\e{}\in\Th{}$, we denote by $\FE\eqbydef\{\f{}\in\Fh\,:\,\f{}\subset\bnd{\e{}} \}$ the set of interfaces composing the boundary of $\e{}$, and we set $\eta_{\e{}}\eqbydef\card{\FE}$. For all elements $X$ of $\Th{}$ or $\Fh{}$, we denote by $\abs{X}$ and $h_{X}$ the measure and diameter of $X$, respectively. 

\subsection{Approximation spaces}

For any polyhedral domain $D\subset\R^d$ with $\partial D\subset\R^{d-1}$, we denote by $\inerprod{\cdot}{\cdot}{D}$ (resp. $\dualprod{\cdot}{\cdot}{\partial D}$) the $\LL$-inner product in $\lebesgue{2}{D}$ (resp. $\lebesgue{2}{\partial D}$) equipped with its natural norm $\norm{\cdot}{0}{D}$ (resp. $\norm{\cdot}{0}{\partial D}$). Similarly, we denote by $\hilbert{s}{D}{}$ the usual Hilbert space of index $s$ on $D$ equipped with its natural norm $\norm{\cdot}{s}{D}$ and seminorm $\seminorm{\cdot}{s}{D}$, respectively. In particular, if $s=0$, then we set $\hilbert{0}{D}{}=\lebesgue{2}{D}$. We now denote by $\hilbert{s}{\Th{}}{}$ the usual broken Sobolev space and by $\grdh$ the broken gradient operator acting on $\hilbert{s}{\Th{}}{}$ with $s\geq 1$. Let us now introduce compact notation associated with the discrete $\LL$-inner scalar product:
\begin{equation}
\inerprod{\cdot}{\cdot}{\Th{}}\eqbydef\sumTh\inerprod{\cdot}{\cdot}{\e{}},\quad\dualprod{\cdot}{\cdot}{\bnd{\Th{}}}\eqbydef\sumTh\dualprod{\cdot}{\cdot}{\bnd{\e{}}}\quad\textrm{and}\quad\dualprod{\cdot}{\cdot}{\Fh{}}\eqbydef\sumFh\dualprod{\cdot}{\cdot}{\f{}},
\end{equation}
and we denote by $\norm{\cdot}{0}{\Th{}}$, $\norm{\cdot}{0}{\bnd{\Th{}}}$ and $\norm{\cdot}{0}{\Fh{}}$ its corresponding norms, respectively. As usual in HDG methods, we consider broken Sobolev spaces:
\begin{equation}
    \Pk{k}{\Th{}}\eqbydef \, \{ \dv\in\lebesgue{2}{\Th{}}\, : \, \dv\vert_{\e{}} \in \Pk{k}{\e{}}, \; \forall \e{}\in\Th{}\},
\end{equation}
and similarly for $\Pk{k}{\Fh{}}$. Here, $\Pk{k}{X}$ denotes the space of polynomials of at least degree $k$ on $X$, where $X$ corresponds to a generic element of $\Th{}$ or $\Fh{}$, respectively. For H-IP discretization, two types of discrete variables are necessary to approximate the weak solution $u$ of problem \eqref{continuous-deg-problem}. First, the discrete variable $\du\in\Vh$ is defined within each mesh element, and its trace $\dtu\in\TVh{0}$ is defined on the mesh skeleton with respect to the imposed Dirichlet boundary conditions at the boundary part $\Gamma^{-}$. Thus, we set: 
\begin{equation}
\Vh\eqbydef\polyn{k}{\Th{}}{}\quad\textrm{and}\quad\TVh{0}\eqbydef\, \{ \dtv\in\polyn{k}{\Fh{}}{}\, : \, \dtv\vert_{\f{}} = 0,  \; \forall\f{}\in\Fhbm  \}.
\end{equation}
For clarity, we introduce compact discrete variables $\cdu\eqbydef (\du,\dtu)$ and $\cdv\eqbydef (\dv,\dtv)$ belonging to the composite approximation space $\CVh{}\eqbydef\Vh\times\TVh{0}$, i.e., $\cdu,\cdv\in\CVh{}$. 
Finally, let us retain the discrete trace inequalities. For all $\dv\in\Vh{}$, the following holds,
\begin{equation}
    \label{discret_trace_ineq}
    \norm{\dv}{0}{\f{}}\leq\,\C{tr}h^{-\ohalf}_{\e{}}\norm{\dv}{0}{\e{}},\\
\end{equation}
where $\C{tr}$ is a positive constant independent of $h_{\e{}}$.

\subsection{Discrete trace operators}

Let $\jump{\cdot}$ denote the standard DG-jump operator as introduced by Brezzi \al in \citep{BMS2004M3AS}. For all $\e{}\in\Th{}$ and $\f{}\in\bnd{\e{}}$, we now define the HDG-jump operator of the composite discrete variable $\cdv\in\CVh{}$ across $\f{}$ as $\tjump{\cdv}_{\e{},\f{}}\eqbydef(\dv\vert_{\f{}}-\dtv\vert_{\f{}})\mathbf{n}_{\e{},\f{}}$, where $\mathbf{n}_{\e{},\f{}}$ denotes the unit normal vector to $\f{}$ pointing out of $\e{}$. To ensure that confusion cannot arise, we voluntary omit the subscripts $\e{}$ and $\f{}$ from the definition, and we simply write $\tjump{\cdv}\eqbydef(\dv-\dtv)\mathbf{n}$. Finally, we define the weighted-average operator denoted $\wmean{\cdot}{\vom}$ and its conjugate $\bmean{\cdot}{\vom}$. For all $\f{}\in\Fhi$ with $\f{}\eqbydef\bnd{\e{1}}\cap\bnd{\e{2}}$ and $v\in\hilbert{s}{\Th{}}{}$ with $s\geq 1$, we set:
\begin{equation}\label{eqn:wmean}	
    \wmean{v}{\vom}\eqbydef\omega_1v_{1}+\omega_2v_{2}\quad\textrm{and}\quad\bmean{v}{\vom}\eqbydef\omega_2v_{1}+\omega_1v_{2},
\end{equation}
where $v_{i}=v\vert_{\e{i},\f{}}$ and $\vom\eqbydef(\omega_1,\omega_2)$ is a double-valued function verifying that weights satisfy $\omega_1+\omega_2=1$. If $\f{}\in\Fhb$, we then assume that $\wmean{v}{\vom}=\bmean{v}{\vom}\eqbydef v$. If $\vom=(1/2,1/2)$, we then recover the classical average operator, and we will omit the subscript $\vom$ in their definitions. These definitions \eqref{eqn:wmean} are also available for any vector-valued function $\vect{v}$.

\section{Hybridizable interior penalty discontinuous Galerkin method}
\label{HIP}

In this section, we describe the primal HDG method for solving the problem \eqref{continuous-deg-problem}. First, we derive it intuitively, and we then propose a compact notation of all of these variants. The consistency and coercivity properties are also discussed in order to ensure the well-posedness of the discrete problem. Finally, we suggest an adaptive strategy for selecting suited penalty parameters with respect to the coercivity requirement for all mentioned regimes. 

\subsection{Intuitive derivation}
 
The discrete formulation of the continuous model problem \eqref{continuous-deg-problem} can be intuitively derived with respect to the three following steps: 
\begin{itemize}
\item[$\blacktriangleright$] \underline{Global weak formulation}: Let $\cdu,\cdv\in\CVh{}$. For all $\e{}\in\Th{}$, multiplying \eqref{primal-ADR} by a test function $\dv$, and integrating by parts over $\e{}$, we obtain a local equation:
\begin{equation}
\label{locweak}
 -\inerprod{\dsg(\du)}{\grdh\dv}{\e{}}+\dualprod{\dtsg(\cdu)\cdot\vect{n}}{\dv}{\bnd{\e{}}}+\inerprod{\mu\du}{\dv}{\e{}}=\inerprod{f}{\dv}{\e{}},
\end{equation}
where $\dsg(\du)\eqbydef-\bkappa\grdh\du+\bbeta \du$ corresponds to the approximation of the total flux on $\e{}$ and $\dtsg(\cdu)$, its trace approximation on $\bnd{\e{}}$ that we will precisely define below. By summing \eqref{locweak} over all elements $\e{}$ of $\Th{}$, we then obtain the global equation,
\begin{equation}
\label{globweak}
 -\inerprod{\dsg(\du)}{\grdh\dv}{\Th{}}+\dualprod{\dtsg(\cdu)\cdot\vect{n}}{\dv}{\bnd{\Th{}}}+\inerprod{\mu\du}{\dv}{\Th{}}=\inerprod{f}{\dv}{\Th{}}.
\end{equation}
Let us precisely state at this stage that $\dtsg(\cdu)$ can be evaluated independently on both sides of a given interface. Consequently, an additional equation needs to be included in \eqref{globweak} to ensure the continuity requirements.

\item[$\blacktriangleright$] \underline{Transmission conditions}: By considering that $\dtv\in\TVh{0}$ vanishes on $\Fhbm$, we impose the global transmission requirements respecting the outflow boundary condition on $\Fhbp$: 
\begin{equation}
\label{transcond}
	\dualprod{\dtsg(\cdu)\cdot\vect{n}}{\dtv}{\bnd{\Th{}}}=\dualprod{(\bbeta\cdot\vect{n})\dtu}{\dtv}{\Fhbp},
\end{equation}
which is the statement of weak continuity of $\dtsg(\cdu)\cdot\vect{n}$ across the mesh skeleton $\Fh{+}$. Thus, substituting \eqref{transcond} into \eqref{globweak} leads to
\begin{equation}
\label{finprob}
  -\inerprod{\dsg(\du)}{\grdh\dv}{\Th{}}+\inerprod{\mu\du}{\dv}{\Th{}}+\dualprod{\dtsg(\cdu)}{\tjump{\cdv}}{\bnd{\Th{}}}+\dualprod{(\bbeta\cdot\vect{n})\dtu}{\dtv}{\Fhbp}=\inerprod{f}{\dv}{\Th{}}.
\end{equation}
\item[$\blacktriangleright$] \underline{Numerical Flux in the HDG sense:} The discrete problem \eqref{finprob} is then closed by setting $\dtsg(\cdu)$ at the element level. For all $\e{}\in\Th{}$, we assume the following simple form: 
\begin{equation}
\label{tracesg}
\dtsg(\cdu)\eqbydef\dsg(\du)+\tau\tjump{\cdu}\quad\textrm{on}\quad\bnd{\e{}},
\end{equation}
where $\tau$ denotes the stabilization penalty parameter that we describe precisely below. 
\end{itemize}
Inserting \eqref{tracesg} into \eqref{finprob} leads to the incomplete scheme of the hybridizable interior penalty discontinuous Galerkin (H-IIP) method. Several variations of this methodology can then be derived by controlling the introduction of additional consistent terms to the discrete formulation \eqref{finprob} that we have summarized below.

\subsection{Compact discrete formulation}

 Thus, the compact formulation of the H-IP method consists of seeking $\cdu\in\CVh{}$ such that 
\begin{equation}
\label{HIP-method}
\dbilin{a^{(\epsilon)}_h}{\cdu}{\cdv}=l(\cdv),\qquad\forall \cdv\in\CVh{},
\end{equation} 
where $l(\cdv)\eqbydef\inerprod{f}{\dv}{\Th{}}$ and the bilinear form $a^{(\epsilon)}_h$ can be linearly decomposed following its diffusive, advective-reactive, and stability contributions, respectively:
\begin{equation}
\label{bilinear_decomp}
\dbilin{a^{(\epsilon)}_h}{\cdu}{\cdv}\eqbydef\dbilin{a^{(\epsilon)}_{\bkappa,h}}{\cdu}{\cdv}+\dbilin{a_{\bbeta,\mu,h}}{\cdu}{\cdv}+\dbilin{s_h}{\cdu}{\cdv}.
\end{equation}
Here, $a^{(\epsilon)}_{\bkappa,h}$ corresponds solely to the H-IP discretization of the diffusive part and is given by:
\begin{equation}
\label{diffusive_form}
\dbilin{a^{(\epsilon)}_{\bkappa,h}}{\cdu}{\cdv}\eqbydef\inerprod{\bkappa\grdh\du}{\grdh\dv}{\Th{}}-\dualprod{\bkappa\grdh\du}{\tjump{\cdv}}{\bnd{\Th{}}}-\epsilon\dualprod{\bkappa\grdh\dv}{\tjump{\cdu}}{\bnd{\Th{}}},
\end{equation}
where the parameter $\epsilon\in\{0,\pm 1\}$ controls the introduction of the (consistent) symmetry term $\dualprod{\bkappa\grdh\dv}{\tjump{\cdu}}{\bnd{\Th{}}}$. Similar to the standard IPDG methods, $\epsilon=0$ corresponds to the Incomplete scheme denoted as the H-IIP method as described above, while $\epsilon=+1$ (resp. $\epsilon=-1$) denotes the Symmetric (resp. Nonsymmetric) scheme denoted as the H-SIP (resp. H-NIP) method. Thus, the advective-reactive part is discretized as follows:
\begin{equation}
\label{advective_form}
\dbilin{a_{\bbeta,\mu,h}}{\cdu}{\cdv}\eqbydef-\inerprod{\bbeta\du}{\grdh\dv}{\Th{}}+\inerprod{\mu\du}{\dv}{\Th{}}+\dualprod{\bbeta\du}{\tjump{\cdv}}{\bnd{\Th{}}}+\dualprod{(\bbeta\cdot\vect{n})\dtu}{\dtv}{\Fhbp}.
\end{equation}
The last quantity $s_h$ in \eqref{bilinear_decomp} is called the discrete stability form based on jump-penalty terms and is given by:
\begin{equation}
    \label{stability_form}
    \dbilin{s_h}{\cdu}{\cdv}\eqbydef\dualprod{\tau\tjump{\cdu}}{\tjump{\cdv}}{\bnd{\Th{}}},
\end{equation}
where $\tau$ denotes the total stabilization function accounting for both the diffusive and advective normal effects denoted by $\tau_{\bkappa}$ and $\tau_{\bbeta}$, respectively; \ie $\tau\eqbydef\tau(\tau_{\bkappa},\tau_{\bbeta})$. Concretely, $\tau_{\bkappa}$ (resp. $\tau_{\bbeta}$) corresponds to an arbitrary given definition of the penalty parameter in the context of a pure-diffusive (resp. -advective) regime. Let us also precisely state that $\tau_{\bkappa}$ and $\tau_{\bbeta}$, and hence $\tau$ can be distinctively evaluated on both sides of an interface of the mesh skeleton. Different formulations have been proposed in the literature, and its definition is concretely prescribed regarding the predominant normal effect. In what follows, we shall assume some properties and minimal requirements concerning the function $\tau$:
\begin{proposition}[]
\label{mini_requirement}
The stabilization function $\tau:\R\times\R\to\R$ \eqref{stability_form} is chosen such that: 
\begin{enumerate}
    \item 
    For all $s,t\in\R$, the following holds:
    \begin{equation}
        \label{first_second_assump}
        s\eqbydef\tau(s,0)\quad\textrm{and}\quad t\eqbydef\tau(0,t).
    \end{equation}
    The first assumption \eqref{first_second_assump} allows for an encompassed treatment of both (extremal) configurations, \ie the pure-diffusive and pure-advective regimes characterized by $\tau_{\bbeta}\eqbydef 0$ and $\tau_{\bkappa}\eqbydef 0$, respectively.
    \item We shall assume the existence of a constant $\tau_0>0$ such that, for all $s,t\in\R$, then
    \begin{equation}
        \label{minimal_tau}
        \min(\tau(s,t)-\tau(s,0)+\dfrac{\bbeta\cdot\vect{n}}{2})\vert_{\bnd{\Th{}}}\geq\tau_0
    \end{equation}
    The second assumption ensures the coercivity and hence the well-posedness of the discrete problem \eqref{HIP-method} (see Lemma \ref{coercivity}).
\end{enumerate}
\end{proposition}

\begin{definition}[Diffusive penalty parameter]
\label{def_diffusive_penalty}
For all $\e{}\in\Th{}$ and $\f{}\in\FE$, we assume that the diffusive parameter $\tau_{\bkappa}\vert_{\e{},\f{}}$ has the following form, 
\begin{equation}
\label{diff_penalty}
\tau_{\bkappa}\vert_{\e{},\f{}}\eqbydef\alpha_0C^2_{\mathrm{tr}}\dfrac{\kappa_{\e{},\f{}}}{h_{\e{}}}\quad\textrm{on}\quad\f{}\in\bnd{\e{}},
\end{equation}
where $\kappa_{\e{},\f{}}\eqbydef\vect{n}_{\e{},\f{}}\bkappa_{\e{}}\vect{n}_{\e{},\f{}}$ corresponds to the normal diffusivity, $\alpha_0>0$ is a user-dependent parameter, and $C_{\mathrm{tr}}$ is the constant of the discrete trace inequality \eqref{discret_trace_ineq}.
\end{definition}


\begin{remark}[Static condensation]
\label{remark2}
As mentioned in \citep{Cockburn2016} (see, \eg \citep{CDEHHO}), the technique of static condensation was introduced to reduce the size of the discrete matrix associated with the global problem \eqref{HIP-method}. Indeed, let $\boldsymbol{\mathrm{U}}_h\eqbydef[\mathrm{U}_h,\hat{\mathrm{U}}_h]$ denote the vector of DOFs of the composite variable $\cdu\in\CVh{}$, which is composed of interior- and interface-based DOFs denoted by $\mathrm{U}_h$ and $\hat{\mathrm{U}}_h$, respectively. The strategy consists of eliminating interior-based unknowns from the above equations by successively projecting \eqref{HIP-method} on $(\dv,0)$ and $(0,\dtv)$. We thus obtain 
\begin{equation}
    \label{static_cond}
    \begin{array}{rcc}
        \dbilin{a^{(\epsilon)}_h}{\cdu}{(\dv,0)}&:&\mathrm{A}_{uu}\mathrm{U}_h+\mathrm{A}_{u\hat{u}}\hat{\mathrm{U}}_h=\mathrm{F},\\
        \dbilin{a^{(\epsilon)}_h}{\cdu}{(0,\dtv)}&:&\mathrm{A}_{\hat{u}u}\mathrm{U}_h+\mathrm{A}_{\hat{u}\hat{u}}\hat{\mathrm{U}}_h=\mathrm{0}.
    \end{array}
\end{equation}
Due to the discontinuous nature of $\Vh$, all computations can be performed cellwise, leading to a block-diagonal matrix $\mathrm{A}_{uu}$, which can be easily inverted and eliminated. Finally, we obtain the linear system:
\begin{equation}
    \label{schur_complement}
    [\mathrm{A}_{\hat{u}\hat{u}}-\mathrm{A}_{\hat{u}u}\mathrm{A}^{-1}_{uu}\mathrm{A}_{u\hat{u}}]\hat{\mathrm{U}}_h=\mathrm{G}.
\end{equation}
The matrix on the left-hand-side of \eqref{schur_complement} is called the Shur complement of $\mathrm{A}_{uu}$. Let us  precisely state that the original discrete problem \eqref{HIP-method} and its reduced version \eqref{schur_complement} are globally nonsymmetric since the continuous problem \eqref{static_cond} is itself nonsymmetric: this is due to the advective term.
\end{remark}

\subsection{Stability analysis}
We now must check some favorable properties such as the consistency and stability of all $\epsilon$-variants of the H-IP method to ensure the existence and uniqueness of a discrete solution. 
\subsubsection{Consistency}
\begin{lemma}[Consistency]
Let $u$ be the exact solution of the problem \eqref{continuous-deg-problem} and $\tu$ denote its trace on the mesh skeleton. By setting $\boldsymbol{u}\eqbydef(u,\tu)$, the following holds
    \begin{equation}
        a^{(\epsilon)}_h(\boldsymbol{u},\cdv)=l(\cdv),\qquad\forall \Dv\in\CVh{},
    \end{equation}
for any value of the parameter $\epsilon=\{0,\pm 1\}$.
\end{lemma}
\begin{proof}
The regularity of $\vect{u}$ implies that its jump (in the HDG sense) is null on $\bnd{\Th{}}$, \ie for all $\e{}\in\Th{}$ and $\f{}\in\FE$, and then $\tjump{\vect{u}}\eqbydef 0$, since $u$ is a single-valued field on $\f{}\in\FE$. Thus, by setting $\cdv\eqbydef(\dv,0)$, integrating by parts on each element of the mesh, and finally observing that $\jump{-\bkappa\grdh u+ \bbeta u}\eqbydef 0$ on internal interfaces, we immediately have
\begin{equation}
\label{local_consist}
 \dbilin{a^{(\epsilon)}_h}{\vect{u}}{(\dv,0)}\eqbydef\sumTh\dbilin{}{\dvgh(-\bkappa\grdh u+ \bbeta u)+\mu u}{\dv}=\sumTh\dbilin{}{f}{\dv}.
\end{equation} 
Considering now that $\cdv\eqbydef(0,\dtv)\in\CVh{}$, we then obtain
\begin{equation}
    \label{trans_consist}
    \dbilin{a^{(\epsilon)}_h}{\vect{u}}{(0,\dtv)}\eqbydef-\sumTh\lspb{(-\bkappa\grdh u+\bbeta u)\cdot\vect{n}}{\dtv}+\lspabp{(\bbeta\cdot\vect{n})u}{\dtv}=0,
\end{equation}
which corresponds to the (imposed) transmission conditions. The proof is then completed by summing \eqref{local_consist} and \eqref{trans_consist}.
\end{proof}

\subsubsection{Coercivity}
\label{coercivity}

Let us now introduce the natural energy-norm $\tnorm{\cdot}$ equipping $\CVh{}$. For all $\cdv\in\CVh{}$, it is thus given by
\begin{equation}
    \label{nrj_norm}
    \tnorm{\cdv}^2\eqbydef\norm{\bkappa^{\ohalf}\grdh{\dv}}{0}{\Th{}}^2+\norm{\mu^{\ohalf}_0\dv}{0}{\Th{}}^2+\norm{\beta^{\ohalf}\dtv}{0}{\Fhbp}^2+\hdgseminorm{\cdv}{\tau_{0}}^2,
\end{equation}
which clearly depends on constitutive coefficients $\bkappa$, $\beta\eqbydef\bbeta\cdot\vect{n}$, and parameters $\mu_{0}$ and $\tau_0$ as given in \eqref{coercivity-cond} and \eqref{minimal_tau}, respectively. Here, $\hdgseminorm{\cdot}{\gamma}$ corresponds to the HDG-jump seminorm, which is defined as follows:
\begin{equation}
    \hdgseminorm{\cdv}{\gamma}^2\eqbydef\sum_{\e{}\in\Th{}}\seminorm{\cdv}{\gamma}{\bnd{\e{}}}^2\quad\textrm{with}\quad\seminorm{\cdv}{\gamma}{\bnd{\e{}}}^2\eqbydef\sum_{\f{}\in\FE}\norm{\gamma^{\ohalf}_{\e{},\f{}}\tjump{\cdv}}{0}{\f{}}^2,
\end{equation}
where $\gamma_{\e{},\f{}}\geq 0$ is an arbitrary positive constant associated with $\f{}\in\FE$. The norm \eqref{nrj_norm} is well-defined owing to assumptions \eqref{coercivity-cond} and \eqref{minimal_tau}. To prove the coercivity, let us now introduce the following intermediate result, which is a minor adaptation to the $\bkappa$-tensor field of the proof given by Wells in \citep{wells2011analysis}:
\begin{lemma}[A bound on consistency term]
\label{lemma_bound}
Assuming \eqref{diff_penalty}, for any $\xi>0$, the following holds,
\begin{equation}
    \abs{\dualprod{\bkappa\grdh\dv}{\tjump{\cdv}}{\bnd{\Th{}}}}\leq\dfrac{\xi\eta_0}{2\alpha_0}\norm{\bkappa^{\ohalf}\grdh\dv}{0}{\Th{}}^2+\dfrac{1}{2\xi}\hdgseminorm{\cdv}{\tau_{\bkappa}}^2,
\end{equation}
where $\eta_{0}\eqbydef\underset{\forall\e{}\in\Th{}}{\max}(\eta_{\e{}})$.
\end{lemma}
\begin{proof}
The decomposition of the consistency term yields
\begin{equation}
    \abs{\dualprod{\bkappa\grdh\dv}{\tjump{\cdv}}{\bnd{\Th{}}}}\leq\sumTh\abs{\dualprod{\bkappa^{\ohalf}\grdh\dv}{\bkappa^{\ohalf}\tjump{\cdv}}{\bnd{\e{}}}}.
\end{equation}
Successively applying the Cauchy--Schwarz inequality and the discrete trace inequality \eqref{discret_trace_ineq}, and using the definition of $\tau_{\bkappa}$ given in \eqref{diff_penalty}, we infer that
\begin{eqnarray}
    \abs{\dualprod{\bkappa^{\ohalf}\grdh\dv}{\bkappa^{\ohalf}\tjump{\cdv}}{\bnd{\e{}}}}&\leq&\bigg[\dfrac{h_{\e{}}}{\alpha_0C_{\mathrm{tr}}^2}\bigg]^{\ohalf}\norm{\bkappa^{\ohalf}\grdh\dv}{0}{\bnd{\e{}}}\norm{\tau_{\bkappa}^{\ohalf}\tjump{\cdv}}{0}{\bnd{\e{}}},\\
    &\leq&\bigg[\dfrac{\eta_{\e{}}}{\alpha_{0}}\bigg]^{\ohalf}\norm{\bkappa^{\ohalf}\grdh\dv}{0}{\e{}}\seminorm{\cdv}{\tau_{\bkappa}}{\bnd{\e{}}},
\end{eqnarray}
The proof is then completed by applying Young's inequality for any $\xi>0$ and summing over all mesh elements.
\end{proof}
\begin{lemma}[Coercivity]
For any value of the parameter $\epsilon\eqbydef\{0,\pm 1\}$, there exists a positive constant $C$ independent of $h$ such that 
    \begin{equation}
        \dbilin{a^{(\epsilon)}_h}{\cdv}{\cdv}\geq C\tnorm{\cdv}^2.
    \end{equation}
\end{lemma}
\begin{proof}
Setting $\cdu=\cdv$ in discrete bilinear forms $a^{(\epsilon)}_{\bkappa,h}$ \eqref{diffusive_form} and $s_h$ \eqref{stability_form}, we immediately obtain that: 
\begin{equation}
\label{difstab_decomp}
\dbilin{a^{(\epsilon)}_{\bkappa,h}}{\cdv}{\cdv}+\dbilin{s_{h}}{\cdv}{\cdv}\eqbydef\norm{\bkappa^{\ohalf}\grdh{\dv}}{0}{\Th{}}^2-(1+\epsilon)\dualprod{\bkappa\grdh\dv}{\tjump{\cdv}}{\bnd{\Th{}}},
+\dualprod{\tau\tjump{\cdv}}{\tjump{\cdv}}{\bnd{\Th{}}}.
\end{equation}
After integration by parts, the advective bilinear form \eqref{advective_form} yields that
\begin{equation}
\label{adv_decomp}
\dbilin{a_{\bbeta,\mu,h}}{\cdv}{\cdv}\eqbydef
\norm{\mu^{\ohalf}_{\ast}\dv}{0}{\Th{}}^2+\norm{\beta^{\ohalf}\dtv}{0}{\Fhbp}^2+\underbrace{\dualprod{\bbeta\dv}{\tjump{\cdv}}{\bnd{\Th{}}}-\frac{1}{2}\dualprod{\beta\dv}{\dv}{\bnd{\Th{}}}}_{=T_1},
\end{equation}
where $\mu_{\ast}\eqbydef\mu+\frac{\nabla\cdot\bbeta}{2}>0$ by virtue of \eqref{coercivity-cond}, and $\beta\eqbydef\bbeta\cdot\vect{n}$. Let us now focus on the (last) quantity $T_1$ in $\eqref{adv_decomp}$. Considering that $\dtv$ and $\beta$ are single-valued on interfaces of the mesh skeleton, and after some tedious algebraic manipulations, we deduce that
\begin{eqnarray}
    T_{1}&=&\dualprod{\beta\tjump{\cdv}}{\tjump{\cdv}}{\bnd{\Th{}}}+\dualprod{\beta\dtv}{\dv-\dtv}{\bnd{\Th{}}}-\frac{1}{2}\dualprod{\beta\dv}{\dv}{\bnd{\Th{}}}\nonumber\\
    &=&\dualprod{\beta\tjump{\cdv}}{\tjump{\cdv}}{\bnd{\Th{}}}-\frac{1}{2}\dualprod{\beta(\dv-\dtv)}{\dv}{\bnd{\Th{}}}+\frac{1}{2}\dualprod{\beta(\dv-\dtv)}{\dtv}{\bnd{\Th{}}}\nonumber\\
    \label{simple_adv_term}
    &=&\frac{1}{2}\dualprod{\beta\tjump{\cdv}}{\tjump{\cdv}}{\bnd{\Th{}}}
\end{eqnarray}
Inserting \eqref{simple_adv_term} in \eqref{adv_decomp}, and finally collecting \eqref{difstab_decomp}, we immediately obtain 
\begin{eqnarray}
\dbilin{a^{(\epsilon)}_h}{\cdv}{\cdv}&\eqbydef&\norm{\bkappa^{\ohalf}\grdh{\dv}}{0}{\Th{}}^2+\norm{\mu^{\ohalf}_{\ast}\dv}{0}{\Th{}}^2+\norm{\beta^{\ohalf}\dtv}{0}{\Fhbp}^2+\dualprod{\tau_{\ast}\tjump{\cdv}}{\tjump{\cdv}}{\bnd{\Th{}}}\nonumber\\
&-&(1+\epsilon)\dualprod{\bkappa\grdh\dv}{\tjump{\cdv}}{\bnd{\Th{}}},
\end{eqnarray}
where $\tau_{\ast}\eqbydef\tau+\frac{\bbeta\cdot\vect{n}}{2}$. Following \eqref{coercivity-cond} and \eqref{minimal_tau}, we deduce that $\mu_{\ast}\geq\mu_{0}$ and $\tau_{\ast}\geq\tau_{0}$, and hence
\begin{equation}
\dbilin{a^{(\epsilon)}_h}{\cdv}{\cdv}\geq\tnorm{\cdv}^2-(1+\epsilon)\dualprod{\bkappa\grdh\dv}{\tjump{\cdv}}{\bnd{\Th{}}},
\end{equation}
proving immediately the coercivity of the H-NIP scheme ($\epsilon=-1$) with $C=1$. Else, by considering Lemmata \ref{lemma_bound} and assuming $0<\xi<1$, we easily infer that 
\begin{eqnarray*}
\dbilin{a^{(\epsilon)}_h}{\cdv}{\cdv}&\geq&(1-\dfrac{\xi\eta_0}{\alpha_0})\norm{\bkappa^{\ohalf}\grdh{\dv}}{0}{\Th{}}^2+\norm{\mu^{\ohalf}_0\dv}{0}{\Th{}}^2+\norm{\beta^{\ohalf}\dtv}{0}{\Fhbp}^2+\hdgseminorm{\cdv}{\tau_{0}}^2,\\
&\geq&C_{\xi}\tnorm{\cdv}
\end{eqnarray*}
where $C_{\xi}\eqbydef\min(1,1-\frac{\xi\eta_0}{\alpha_0})$. The proof is ended by choosing $\alpha_0$ such that $\alpha_0>\xi\eta_0$.
\end{proof}

\begin{remark}
A straightforward consequence of the consistency and coercivity requirements via the Lax--Milgram Theorem is the well-posedness of \eqref{HIP-method}; \ie the existence and uniqueness of the discrete solution $\cdu\in\CVh{}$ are ensured.
\end{remark}

\subsection{Adaptive stabilization strategy}

In practice, the choice of $\tau$ is quite delicate, as it strongly affects the accuracy of the HDG method \eqref{HIP-method}. Indeed, its definition directly impacts the numerical flux approximations on interfaces $\f{}\in\Fh{+}$. To prove its relevance, let us apply a continuity argument, \ie $\jump{\dtsg(\cdu)}=0$, on an interior interface $\f{}\eqbydef\bnd{\e{1}}\cap\bnd{\e{2}}$. We immediately deduce that $(\dtu,\dtsg)$ can be expressed only in terms of the discrete variables $(\du,\dsg)$ on both sides of $\f{}$,
\begin{subequations}\label{eqn:tracesdefinition}
    \begin{empheq}[left=\empheqlbrace]{align}
 		\dtu &=\wmean{\du}{\vom}+\alpha\jump{\dsg},\label{subeqn-1:tracepotential}\\
		\dtsg &=\bmean{\dsg}{\vom}+\eta\jump{\du},\label{subeqn-2:traceflux}
     \end{empheq}
   \end{subequations}
where $\dsg\eqbydef-\bkappa\grdh\du+\bbeta\du$, and the parameters $\vom$, $\alpha$ and $\eta$ are given below by:
\begin{equation}
\label{weightfunc}
\vom\eqbydef\bigg(\dfrac{\tau_1}{\tau_1+\tau_2},\dfrac{\tau_2}{\tau_1+\tau_2}\bigg),\quad \alpha\eqbydef\dfrac{1}{\tau_1+\tau_2},\quad\textrm{and}\quad\eta\eqbydef\dfrac{\tau_1\tau_2}{\tau_1+\tau_2},
\end{equation}
for any given finite value $\tau_{i}\eqbydef\tau_{\e{i},\f{}}$ of the total penalty parameter. To derive a suitable analytical expression of $\tau$, we now treat both hyperbolic and elliptic regimes distinctively. 

\subsubsection{Hyperbolic regime}

We assume here that $\tau_{\bkappa}\eqbydef 0$, and hence $\tau\eqbydef\tau_{\bbeta}$ by virtue of \eqref{first_second_assump}. To ensure the minimal requirement \eqref{minimal_tau} in the context of the hyperbolic regime, the advective penalty parameter must be chosen such that $\tau_{\bbeta}>\abs{\bbeta\cdot\vect{n}}/2$ on the mesh skeleton. Furthermore, it is well known that an arbitrary choice of the stability parameter $\tau_{\bbeta}$ can be detrimental in the context of pure-advective problems: discontinuities in the boundary data may trigger large spurious oscillations in the numerical solution. However, these drawbacks can be easily circumvented by adopting an \textit{upwind-based} strategy. To this aim, we now consider the following definition of the advective penalty parameter:
\begin{definition}[$\theta$-upwind penalty]
\label{up_tau}
For all $\e{}\in\Th{hyp}$ and $\f{}\in\bnd{\e{}}$, we assume the following definition of the advective stabilization penalty parameter: 
\begin{equation}
    \label{theta_tau}
    \tau^{\theta}_{\bbeta_{\e{},\f{}}}\eqbydef\theta\abs{\bbeta\cdot\vect{n}_{\e{},\f{}}}\quad\textrm{on}\quad\f{}\in\bnd{\e{}},
\end{equation}
where $\theta> 1/2$ in order to ensure Proposition \ref{mini_requirement}.
\end{definition}
Following from the regularity of the Darcy field $\bbeta$, \ie $\jump{\bbeta}=0$, we infer that the advective penalty parameter as defined in \eqref{theta_tau} is single-valued on interior interfaces of the hyperbolic region. Let us now introduce the \textit{signum function}, which is given as follows:
\begin{equation}
    \sgn{x}\eqbydef\dfrac{x}{\abs{x}}=
    \begin{cases}
 	    -1 & \textrm{if}\quad x< 0,\\
 		0 & \textrm{if}\quad x=0,\\
 		+1 &  \textrm{if}\quad x>0.
     \end{cases}
\end{equation}
\begin{proposition}[$\theta$-upwind fluxes]
Following Definition \ref{up_tau}, for all $\e{1},\e{2}\in\Th{hyp}$ and $\f{}\eqbydef\bnd{\e{1}\cap\bnd{\e{2}}}$, the corresponding numerical fluxes $(\dtu^{\theta},\dtsg^{\theta})$ on $\f{}$ are given by
\begin{subequations}\label{eqn:traces_adv_theta}         
    \begin{empheq}[left=\empheqlbrace]{align}
 		\dtu^{\theta} &\eqbydef\wmean{\du}{}+\dfrac{\bbeta}{2\tau^{\theta}_{\bbeta}}\jump{\du}= \wmean{\du}{\vom_{\theta}},\label{subeqn-1:tracepotential_theta}\\
		\dtsg^{\theta} &\eqbydef\bbeta\bmean{\du}{}+\dfrac{\tau^{\theta}_{\bbeta}}{2}\jump{\du}=\bbeta\bmean{\du}{\vom^{\theta}},\label{subeqn-2:traceflux_theta}
    \end{empheq}
\end{subequations}
where $\vom_{\theta}\eqbydef(\frac{1}{2}+\frac{\sgn{\bbeta\cdot\vect{n}_1}}{2\theta},\frac{1}{2}+\frac{\sgn{\bbeta\cdot\vect{n}_2}}{2\theta})$ and $\vom^{\theta}\eqbydef(\frac{1}{2}+\frac{\theta}{2\sgn{\bbeta\cdot\vect{n}_1}},\frac{1}{2}+\frac{\theta}{2\sgn{\bbeta\cdot\vect{n}_2}})$.
\end{proposition}
\begin{proof}
The proof is evident by substituting the definition \eqref{theta_tau} in \eqref{eqn:tracesdefinition}, and assuming that $\bkappa$ is null in $\Th{hyp}$ and that $\bbeta$ is single-valued on the mesh skeleton. 
\end{proof}
\begin{remark}[Traditional schemes]
\label{traditional-schemes}
By appropriately selecting the value of $\theta$ in \eqref{theta_tau}, we can then establish bridges with some well-known stabilization schemes. 
\begin{itemize}
    \item \textbf{Upwind-scheme}: By setting $\theta=1$, we can observe that $\vom_{1}=\vom^{1}$, and we recover the standard definition of upwinding fluxes denoted by $(\dtu^{\mathrm{up}},\dtsg^{\mathrm{up}})$,
    \begin{equation}
    \label{eqn:upwindtraces}
    \dtu^{\mathrm{up}}\eqbydef\begin{cases}
 		   		u_i & \textrm{if}\quad\bbeta\cdot\vect{n}_i>0,\\
 		 			\mean{\du}	& \textrm{if}\quad\bbeta\cdot\vect{n}_i=0,
     			\end{cases}\quad\textrm{and}\quad\dtsg^{\mathrm{up}}\eqbydef\bbeta\dtu^{\mathrm{up}}.
    \end{equation}
    The corresponding upwind penalty parameter is denoted $\tau^{\mathrm{up}}_{\bbeta}\eqbydef\abs{\bbeta\cdot\vect{n}}$.
    \item \textbf{Centered-scheme}: Assuming now $\theta\to +\infty$, we then obtain the centered fluxes $(\dtu^{\mathrm{c}},\dtsg^{\mathrm{c}})$,  
    \begin{equation}
        \dtu^{\mathrm{c}}\eqbydef\lim_{\theta\to+\infty}\wmean{\du}{\vom_{\theta}}=\mean{\du}\quad\textrm{and}\quad\dtsg^{\mathrm{c}}\eqbydef\bbeta\bmean{\du}{}+\dfrac{\tau^{\infty}_{\bbeta}}{2}\jump{\du}.
    \end{equation}
    However, this situation will be precluded in the rest of the paper as it consists of assigning an infinite value to the penalty parameter since $\tau^{\infty}_{\bbeta}=\lim_{\theta\to+\infty}\tau^{\theta}_{\bbeta}=+\infty$. This choice significantly reduces the accuracy of the discrete solution since it converges to the discrete solution produced by the standard conforming Galerkin method characterized by spurious oscillations. 
\end{itemize}
\end{remark}
\begin{remark}[Degenerate outflow boundaries]
\label{degenerate-out-bound}
Let us finally precisely denote the transmission conditions at degenerative outflow boundaries that belong to the hyperbolic region. For all $\f{}\in\Fhbp$, we impose that $\dtsg(\cdu)\cdot\vect{n}=(\bbeta\cdot\vect{n})\dtu$, where $\dtsg(\cdu)\eqbydef\bbeta\du+\tau_{\bbeta}\tjump{\cdu}$. By combining these expressions, we observe that the role of $\tau
_{\bbeta}$ on outflow boundaries is clearly insignificant, since for any finite value of $\tau_{\bbeta}>0$, we then obtain $\dtu\eqbydef\du$.
\end{remark}
Thereafter, we shall assume that the advective stabilization penalty parameter $\tau_{\bbeta}$ is chosen accordingly with the $\theta$-upwind strategy described in Definition \ref{up_tau}.

\subsubsection{Elliptic regime}
\label{sub-elliptic-regime}

We assume here that the diffusion is not degenerate, \ie $\underbar{$\kappa$}>0$. To ensure the minimal requirement \eqref{minimal_tau} in the context of the elliptic regime, the total penalty parameter is chosen such that $\tau\geq\tau_{\bkappa}+\abs{\bbeta\cdot\vect{n}}/2$. Let us now introduce the P\'eclet number to locally characterize the regime of the physical process. 
\begin{definition}
For all $\e{}\in\Th{ell}$ and $\f{}\in\bnd{\e{}}$, we define the local (oriented) P\'eclet number as follows:
\begin{equation}
    \label{Peclet_number}
    \peclet{\e{},\f{}}^{\theta}\eqbydef\theta\,\dfrac{\bbeta\cdot\vect{n}_{\e{},\f{}}}{\tau_{\bkappa_{\E{},\f{}}}}\quad\textrm{on}\quad\f{}\in\bnd{\e{}},
\end{equation}
where $\theta$ is a positive constant that we will precisely define below. 
\end{definition}
Concretely, the regime is considered as locally (i) diffusion-dominated if $\abs{\peclet{\e{},\f{}}^{\theta}}\sim\theta$, and (ii) advection-dominated if $\abs{\peclet{\e{},\f{}}^{\theta}}\sim+\infty$. Following the definition \eqref{Peclet_number}, we emphasize here that $\peclet{}^{\theta}$ can be distinctively evaluated on both sides of an interface $\f{}\in\Fhi$. For all $\e{}\in\Th{ell}$ and $\f{}\in\bnd{\e{}}$, we now assume the following generic form of the (total) stabilization penalty parameter \ie 
\begin{equation}
    \tau_{\e{},\f{}}\eqbydef\tau_{\bkappa_{\e{},\f{}}}\abs{\mathcal{A}}(\peclet{\e{},\f{}}^{\theta}),
\end{equation}
where $\abs{\mathcal{A}}:\R\rightarrow\R$ is a given function respecting the following minimal requirements:
\begin{itemize}
    \item[(R1)] $\abs{\mathcal{A}}$ must be an even convex function, \ie $\forall s\in\R$ then $\abs{\mathcal{A}}(-s)=\abs{\mathcal{A}}(s)$ respecting $\abs{\mathcal{A}}(0)=1$. The latter condition allows recovery of the definition of the penalty term in the pure diffusive limit, \ie if $\peclet{\e{},\f{}}^{\theta}=0$, then $\tau_{\e{},\f{}}\eqbydef\tau_{\bkappa_{\e{},\f{}}}$. This is in accordance with the first assumption \eqref{first_second_assump} as described in Proposition \ref{mini_requirement}.
    \item[(R2)] The function $\abs{\mathcal{A}}$ must respect the following asymptotic behaviors:
    \begin{equation}
        \lim_{\abs{s}\to\infty}\dfrac{\abs{\mathcal{A}}(s)}{s}=\sgn{s},
    \end{equation}
    which is in accordance with the parity requirement (R1). It allows recovery of the upwinding-based strategy for purely hyperbolic problems, and consequently the definition of the corresponding fluxes \eqref{eqn:upwindtraces} for the advection-dominated regime:
    \begin{equation}
        \tau_{\e{},\f{}}\eqbydef\theta\bbeta\cdot\vect{n}_{\e{},\f{}}\lim_{\abs{\peclet{\e{},\f{}}^{\theta}}\to\infty}\dfrac{\abs{\mathcal{A}}(\peclet{\e{},\f{}}^{\theta})}{\peclet{\e{},\f{}}^{\theta}}=\theta\abs{\bbeta\cdot\vect{n}_{\e{},\f{}}}.
    \end{equation}
     This is in accordance with the second assumption \eqref{first_second_assump} as described in Proposition \ref{mini_requirement}.
     \item[(R3)] There exists $\theta_0\in\R$ such that for all $\theta\geq\theta_0$ and $s\in\R$, then the following holds: \begin{equation}
        \abs{\mathcal{A}}(\theta s)\geq 1+\dfrac{\abs{s}}{2},
     \end{equation}
     which is in accordance with the convexity requirement (R1). This argument is crucial to ensure that $a^{(\epsilon)}_h$ is $\CVh{}$-coercive.
\end{itemize}
Following the above hypotheses of $\abs{\mathcal{A}}$, we can recover some well-known stabilization strategies: 
\begin{itemize}
    \item[$\blacktriangleright$] \textit{The Additive scheme}: By setting $\abs{\mathcal{A}}_{\textrm{add}}(s)\eqbydef 1+\abs{s}$, we easily infer that
    \begin{equation}
        \tau_{\e{},\f{}}\eqbydef\tau_{\bkappa_{\e{},\f{}}}(1+\abs{\peclet{\e{},\f{}}^{\theta}})=\tau_{\bkappa_{\e{},\f{}}}+\tau^{\theta}_{\bbeta_{\e{},\f{}}},
    \end{equation}
    where $\tau^{\theta}_{\bbeta_{\e{},\f{}}}\eqbydef\theta\abs{\bbeta\cdot\vect{n}_{\e{},\f{}}}$ corresponds to the advective stability parameter chosen in accordance with the $\theta$-upwind strategy. Here, we assume that the total stability parameter is simply equal to the sum of its distinctive diffusive and advective contributions, respectively. It is evident that $\abs{\mathcal{A}}_{\textrm{add}}$ respects all above criteria (R1-R3) assuming $\theta\geq 1/2$.
    \item[$\blacktriangleright$] \textit{The Scharfetter--Gummel scheme}: By setting  $\abs{\mathcal{A}}_{\textrm{sg}}(s)\eqbydef\mathcal{B}(-\abs{s})$ where
    \begin{equation}
        \mathcal{B}(s)\eqbydef\begin{cases}
            \dfrac{s}{e^{s}-1} & \textrm{if}\quad s\neq 0,\\
            1 & \textrm{else}.
        \end{cases}
    \end{equation}
    denotes the (well-known) Bernoulli function. Let us notice that the following holds 
    \begin{equation}
        1+\dfrac{\abs{s}}{2}\leq\mathcal{B}(-\abs{s})\leq 1+\abs{s},\quad\forall s\in\R.
    \end{equation}
    Thus, by using a scaling argument, we can easily infer that for all $\theta\in\R$,
    \begin{equation}
        \abs{\mathcal{A}}_{\textrm{add}}(\theta s/2)\leq\abs{\mathcal{A}}_{\textrm{sg}}(\theta s)\leq\abs{\mathcal{A}}_{\textrm{add}}(\theta s),\quad\forall s\in\R,
    \end{equation}
    proving immediately the last condition (R3) for any given $\theta\geq 1$. In practice, the SG-scheme is particularly interesting since it introduces less artificial diffusion than the Add-scheme in the diffusion-dominated regime.  
\end{itemize}
An illustration of both stabilization functions $\abs{\mathcal{A}}_{\textrm{add}}(\theta s)$ and $\abs{\mathcal{A}}_{\textrm{sg}}(\theta s)$ is given below in \figref{fig:AD_SG_functions} using different admissible values of $\theta$.\par
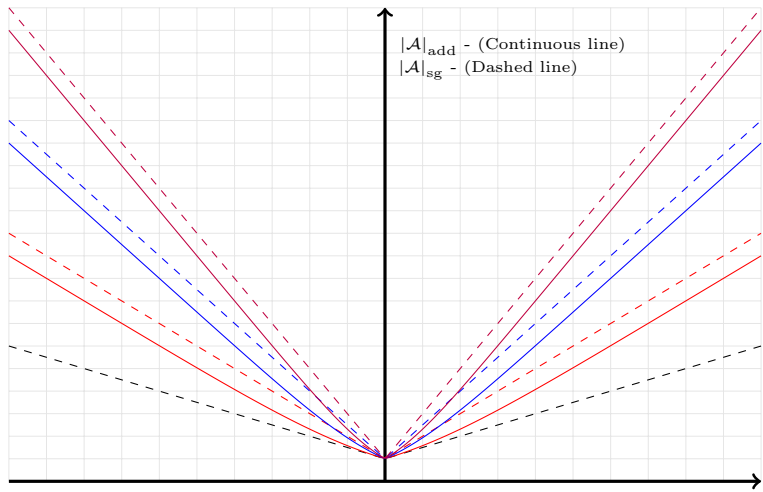
\begin{figure}[!ht]
    \centering   
    \begin{tikzpicture}[domain=-10:10,xscale=0.5,yscale=0.3]
    
        \draw[step=1cm, gray, very thin,opacity=0.2] (-10, 0) grid (10, 21);
        \draw[very thick, ->] (-10, 0) -- (10, 0);
        \draw[very thick, ->] (0,0) -- (0,21);
        \draw[domain= -10:10,samples=200,dashed,color=black,scale=1] plot ({\x},{1+abs(0.5*\x)/1});
        \draw[domain= -10:10,samples=200,color=red,scale=1] plot ({\x},{-abs(\x)/(exp(-abs(\x))-1)});
        \draw[domain=-10:10,samples=200,dashed, color=red,scale=1] plot ({\x},{1+abs(1*\x)});
        \draw[domain= -10:10,samples=200,color=blue,scale=1] plot ({\x},{-abs(1.5*\x)/(exp(-abs(1.5*\x))-1)});
        \draw[domain=-10:10,samples=200,dashed, color=blue,scale=1] plot ({\x},{1+abs(1.5*\x)});
        \draw[domain=-10:10,samples=200,dashed, color=purple,scale=1] plot ({\x},{1+abs(2*\x)});
        \draw[domain= -10:10,samples=200,color=purple,scale=1] plot ({\x},{-abs(2*\x)/(exp(-abs(2*\x))-1)});
        \draw (3.4,19.3) node {\tiny{$\abs{\mathcal{A}}_{\textrm{add}}$ - (Continuous line)}};
        \draw (2.75,18.25) node {\tiny{$\abs{\mathcal{A}}_{\textrm{sg}}$ - (Dashed line)}};
    \end{tikzpicture}
    \caption{An illustration of both stabilization functions $\abs{\mathcal{A}}_{\textrm{add}}(\theta s)$ (dashed-line) and $\abs{\mathcal{A}}_{\textrm{sg}}(\theta s)$ (continuous-line) for the set of values $\theta\eqbydef\{1/2,1,3/2,2\}$ in black, red, blue and purple, respectively. }
    \label{fig:AD_SG_functions}
\end{figure}
 
\section{Numerical results}
\label{Numerical_exp}

To close this section, we provide some numerical experiments illustrating the robustness and accuracy of the proposed H-IP method for solving degenerate advection-diffusion-reaction model problems. For this aim, we next focus on three distinctive physical situations, namely, nondegenerate, fully degenerate and locally degenerate problems, respectively. All of our numerical experiments are performed using the high-performance finite element library called \textsc{NGSolve} \citep{Schoberl2014c++}, and all developed source codes are available for free download from the following Github repository\footnote{\url{https://github.com/GregoryETANGSALE/HDG-Degenerate-ADR-equation}}. 

\subsection{Nondegenerate problem}
We assume here that both $\bkappa$ and $\bbeta$ are nonnull for the whole domain $\Domain\eqbydef[0,1]^2$. In this first example, as proposed by Egger and Schöberl in \citep{egger2009mixed}, we illustrate the ability of the proposed H-IP formalism to deal (efficiently and automatically) with physical processes characterized by a large range of P\'eclet numbers. Both the velocity field $\bbeta$ and the dispersion matrix $\bkappa$ are supposed to be constant on $\Domain$ -- \ie $\bbeta\eqbydef(\beta_1,\beta_2)$ with $\beta_{1,2}\in\R$ and $\bkappa\eqbydef\kappa \vect{I_2}$, where $\vect{I_d}$ denotes the identity matrix in $\R^{d,d}$ and $\kappa > 0$. The exact solution is given by
\begin{equation}
    u(x,y) = [x+ (e^{\beta_1 x / \kappa}-1)/(1- e^{\beta_1 / \kappa}) ] \cdot [y+ (e^{\beta_2 y / \kappa}-1)/(1- e^{\beta_2 / \kappa}) ],
\end{equation}
and the right-hand-side $f$ is chosen such that the exact solution $u$ is verified. For our numerical study, we then set $\bbeta\eqbydef(2,1)$, and we select a large range of $\kappa\eqbydef \{ 5e-1, 5e-2, 5e-3 \}$ to control the ratio between the diffusive and advective contributions. The exact solution $u$ displays sharper fronts on the top and right boundaries of $\Domain$ as $\kappa$ becomes smaller. Since the diffusive part is nonnull inside $\Domain$, we can distinguish three variants of the H-IP method: the H-IIP, H-NIP and H-SIP schemes, respectively. Here, we analyze the influence of the stabilization strategy, namely, the Add- or SG-schemes as defined in Section \ref{sub-elliptic-regime}, on the behavior of discrete solutions. For clarity of our exposition, we set $\theta\eqbydef 1$. Standard $h$- and $k$-refinement strategies are used to compute the discrete $\LL$-errors and estimated convergence rates (ECRs). A history of convergence of the three variants is presented in Table \ref{tableA} for different values of $\kappa\eqbydef\{5e-1,5e-2\}$ and polynomial degrees $k\eqbydef\{1,2\}$, and for both Add- and SG-schemes, respectively. First, a brief analysis indicates that the H-IIP and H-NIP schemes behave differently from the H-SIP scheme regardless of the stabilization function. Their convergence orders are (strongly) influenced by the polynomial parity of $k$. We observe that the convergence rate is suboptimal (with order $k$) only for even $k$, and optimal (with order $k+1$) for odd $k$. The situation is somewhat different for the H-SIP method, since it always converges optimally for all $k$. These statements agree with the theoretical results established by Shin \al in \citep{SHIN2015292} (see \eg \citep{FabienNM2019} for the pure diffusive problem). To pursue our comparative analysis, we then illustrate in Figures \ref{Test-non-degenerate}b and \ref{Test-non-degenerate}d the spatial distribution of the computed $\LL$-error on the whole domain for both Add- and SG-schemes. These illustrations indicate that the SG-scheme produces much less artificial diffusion error than the Add-scheme and hence captures sharp fronts much more effectively. Thus, we will favor the H-SIP variant coupled with the Scharfetter-Gummel scheme in elliptic regions in all future experiments.\par 
\begin{table}[!ht]
\centering
\small
\begin{tabular}{|cc|cccc|cccc|}
\hline
\multirow{3}{*}{$k$} & \multirow{3}{*}{$h^{-1}$} & \multicolumn{4}{c|}{Additive-Upwind} & \multicolumn{4}{c|}{Scharfetter-Gummel} \\
& & \multicolumn{2}{c}{$\kappa = 5e{-1}$} & \multicolumn{2}{c|}{$\kappa = 5e{-2}$} & \multicolumn{2}{c}{$\kappa = 5e{-1}$} & \multicolumn{2}{c|}{$\kappa = 5e{-2}$} \\
& & $\norm{u-\du}{0}{\Th{}}$ & ECR & $\norm{u-\du}{0}{\Th{}}$ & ECR & $\norm{u-\du}{0}{\Th{}}$ & ECR & $\norm{u-\du}{0}{\Th{}}$ & ECR \\
\hline
\hline
& & \multicolumn{8}{c|}{H-NIP} \\
\hline
\multirow{5}{*}{$1$} 
& $4$ & $2.2e-03$ & -- & $9.7e-02$ & -- & $2.1e-03$ & -- & $8.3e-02$ & --   \\
& $8$ & $5.4e-04$ & $2.05$ & $4.1e-02$ & $1.24$ & $5.3e-04$ & $2.02$ & $3.7e-02$ & $1.18$ \\
& $16$ & $1.3e-04$ & $2.03$ & $1.3e-02$ & $1.72$ & $1.3e-04$ & $2.01$ & $1.2e-02$ & $1.67$ \\
& $32$ & $3.3e-05$ & $2.02$ & $3.1e-03$ & $2.01$ & $3.2e-05$ & $2.01$ & $3.0e-03$ & $1.97$ \\
& $64$ & $8.1e-06$ & $2.01$ & $7.4e-04$ & $2.06$ & $8.1e-06$ & $2.00$ & $7.2e-04$ & $2.03$ \\
\hline
\multirow{5}{*}{$2$} 
& $4$ & $4.0e-04$ & -- & $5.1e-02$ & -- & $4.0e-04$ & -- & $5.0e-02$ & -- \\
& $8$ & $8.8e-05$ & $2.16$ & $1.7e-02$ & $1.61$ & $8.9e-05$ & $2.17$ & $1.7e-02$ & $1.56$ \\
& $16$ & $2.1e-05$ & $2.05$ & $4.2e-03$ & $1.99$ & $2.1e-05$ & $2.06$ & $4.3e-03$ & $1.98$ \\
& $32$ & $5.3e-06$ & $2.01$ & $9.4e-04$ & $2.15$ & $5.3e-06$ & $2.01$ & $9.6e-04$ & $2.16$ \\
& $64$ & $1.3e-06$ & $2.00$ & $2.2e-04$ & $2.09$ & $1.3e-06$ & $2.00$ & $2.2e-04$ & $2.10$ \\
\hline
& & \multicolumn{8}{c|}{H-IIP} \\
\hline
\multirow{5}{*}{$1$} 
& $4$ & $2.2e-03$ & -- & $9.7e-02$ & -- & $2.1e-03$ & -- & $8.3e-02$ & -- \\
& $8$ & $5.4e-04$ & $2.05$ & $4.1e-02$ & $1.24$ & $5.3e-04$ & $2.02$ & $3.7e-02$ & $1.18$ \\
& $16$ & $1.3e-04$ & $2.03$ & $1.2e-02$ & $1.72$ & $1.3e-04$ & $2.01$ & $1.2e-02$ & $1.67$ \\
& $32$ & $3.3e-05$ & $2.02$ & $3.1e-03$ & $2.01$ & $3.2e-05$ & $2.01$ & $3.0e-03$ & $1.97$ \\
& $64$ & $8.1e-06$ & $2.01$ & $7.4e-04$ & $2.06$ & $8.1e-06$ & $2.00$ & $7.2e-04$ & $2.03$ \\
\hline
\multirow{5}{*}{$2$} 
& $4$ & $2.7e-04$ & -- & $4.6e-02$ & -- & $2.7e-04$ & -- & $4.3e-02$ & -- \\
& $8$ & $5.2e-05$ & $2.37$ & $1.4e-02$ & $1.75$ & $5.3e-05$ & $2.36$ & $1.3e-02$ & $1.70$ \\
& $16$ & $1.2e-05$ & $2.14$ & $3.0e-03$ & $2.18$ & $1.2e-05$ & $2.14$ & $3.0e-03$ & $2.15$ \\
& $32$ & $2.9e-06$ & $2.04$ & $5.9e-04$ & $2.35$ & $2.9e-06$ & $2.04$ & $5.9e-04$ & $2.34$ \\
& $64$ & $7.2e-07$ & $2.01$ & $1.3e-04$ & $2.23$ & $7.2e-07$ & $2.01$ & $1.3e-04$ & $2.23$ \\
\hline
& & \multicolumn{8}{c|}{H-SIP} \\
\hline
\multirow{5}{*}{$1$} 
& $4$ & $2.1e-03$ & -- & $6.9e-02$ & -- & $2.1e-03$ & -- & $8.3e-02$ & -- \\
& $8$ & $5.2e-04$ & $2.02$ & $3.3e-02$ & $1.06$ & $5.3e-04$ & $2.03$ & $3.7e-02$ & $1.17$ \\
& $16$ & $1.3e-04$ & $2.01$ & $1.1e-02$ & $1.59$ & $1.3e-04$ & $2.02$ & $1.2e-02$ & $1.66$ \\
& $32$ & $3.2e-05$ & $2.01$ & $2.9e-03$ & $1.93$ & $3.2e-05$ & $2.01$ & $3.0e-03$ & $1.97$ \\
& $64$ & $8.1e-06$ & $2.00$ & $7.2e-04$ & $2.01$ & $8.1e-06$ & $2.00$ & $7.2e-04$ & $2.03$ \\
\hline
\multirow{5}{*}{$2$} 
& $4$ & $1.6e-04$ & -- & $3.5e-02$ & -- & $1.7e-04$ & -- & $3.8e-02$ & -- \\
& $8$ & $2.1e-05$ & $2.95$ & $1.0e-02$ & $1.78$ & $2.1e-05$ & $2.95$ & $1.1e-02$ & $1.85$ \\
& $16$ & $2.7e-06$ & $2.99$ & $1.9e-03$ & $2.40$ & $2.7e-06$ & $2.99$ & $2.0e-03$ & $2.44$ \\
& $32$ & $3.4e-07$ & $3.00$ & $2.8e-04$ & $2.77$ & $3.4e-07$ & $3.00$ & $2.9e-04$ & $2.79$ \\
& $64$ & $4.2e-08$ & $3.00$ & $3.7e-05$ & $2.93$ & $4.2e-08$ & $3.00$ & $3.7e-05$ & $2.94$ \\
\hline
\end{tabular}
\caption{Test A - History of convergence of the H-NIP, H-IIP and H-SIP methods on uniform square meshes using the Additive-Upwind and Scharfetter-Gummel scheme with $\kappa = \{ 5e{-1}, 5e{-2} \}$.}
\label{tableA}
\end{table}

\begin{figure}[!ht]
\centering
\subfloat[]{\includegraphics[scale=0.26,trim = 555 120 450 130, clip=true]{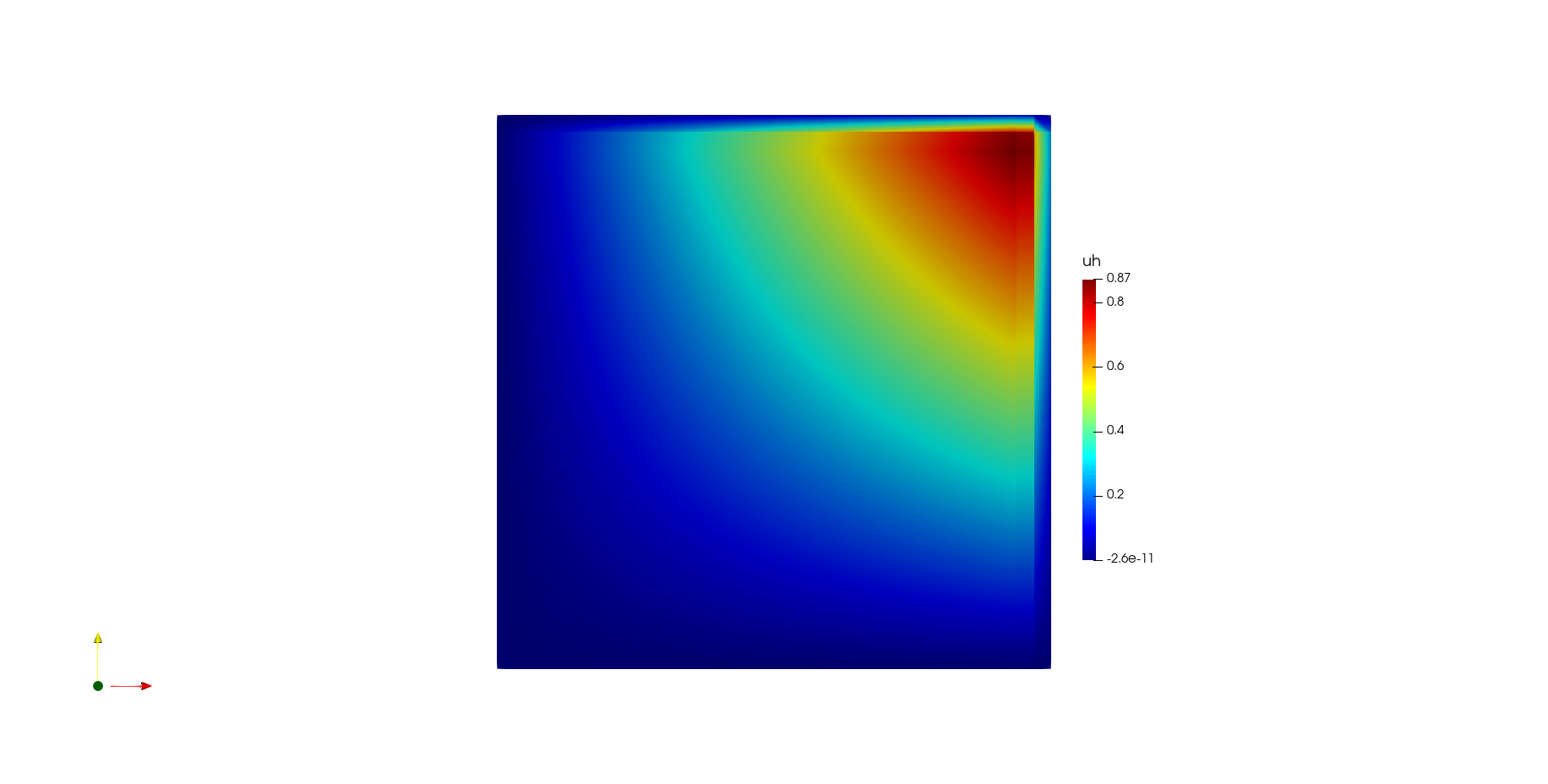}}  
\subfloat[]{\includegraphics[scale=0.26,trim = 555 120 450 130, clip=true]{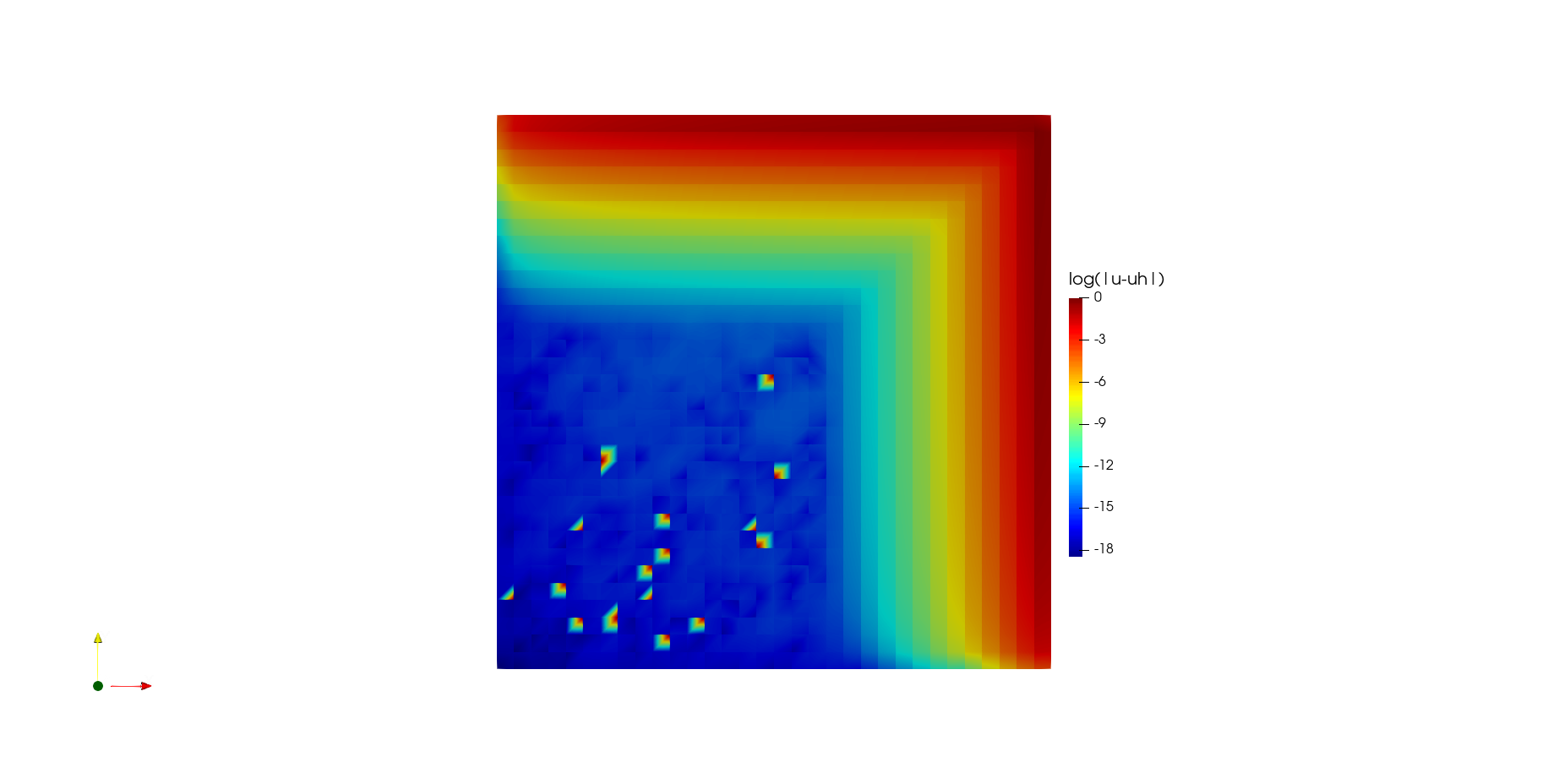}} \\
\subfloat[]{\includegraphics[scale=0.26,trim = 555 120 450 130, clip=true]{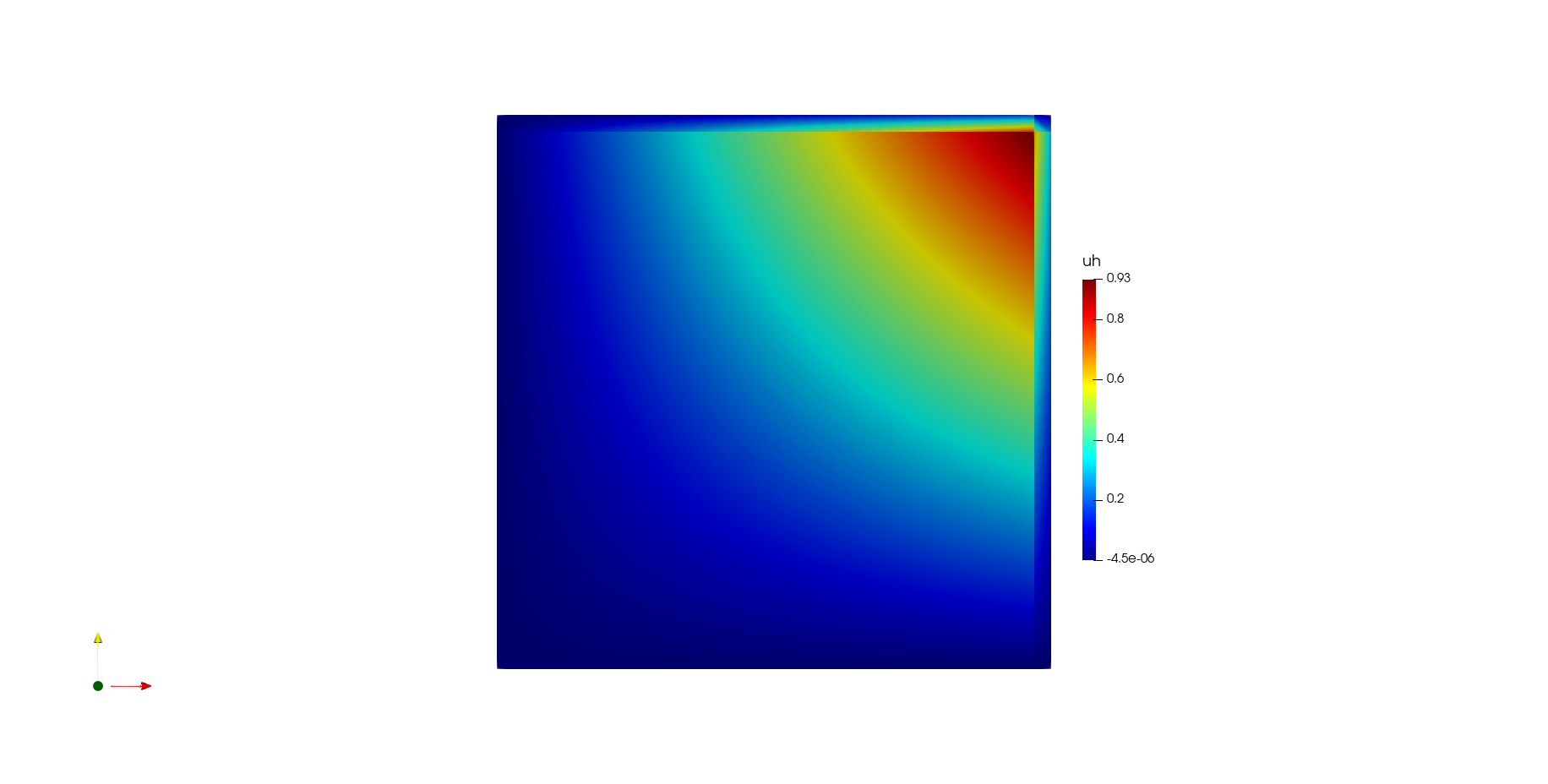}}
\subfloat[]{\includegraphics[scale=0.26,trim = 555 120 450 130, clip=true]{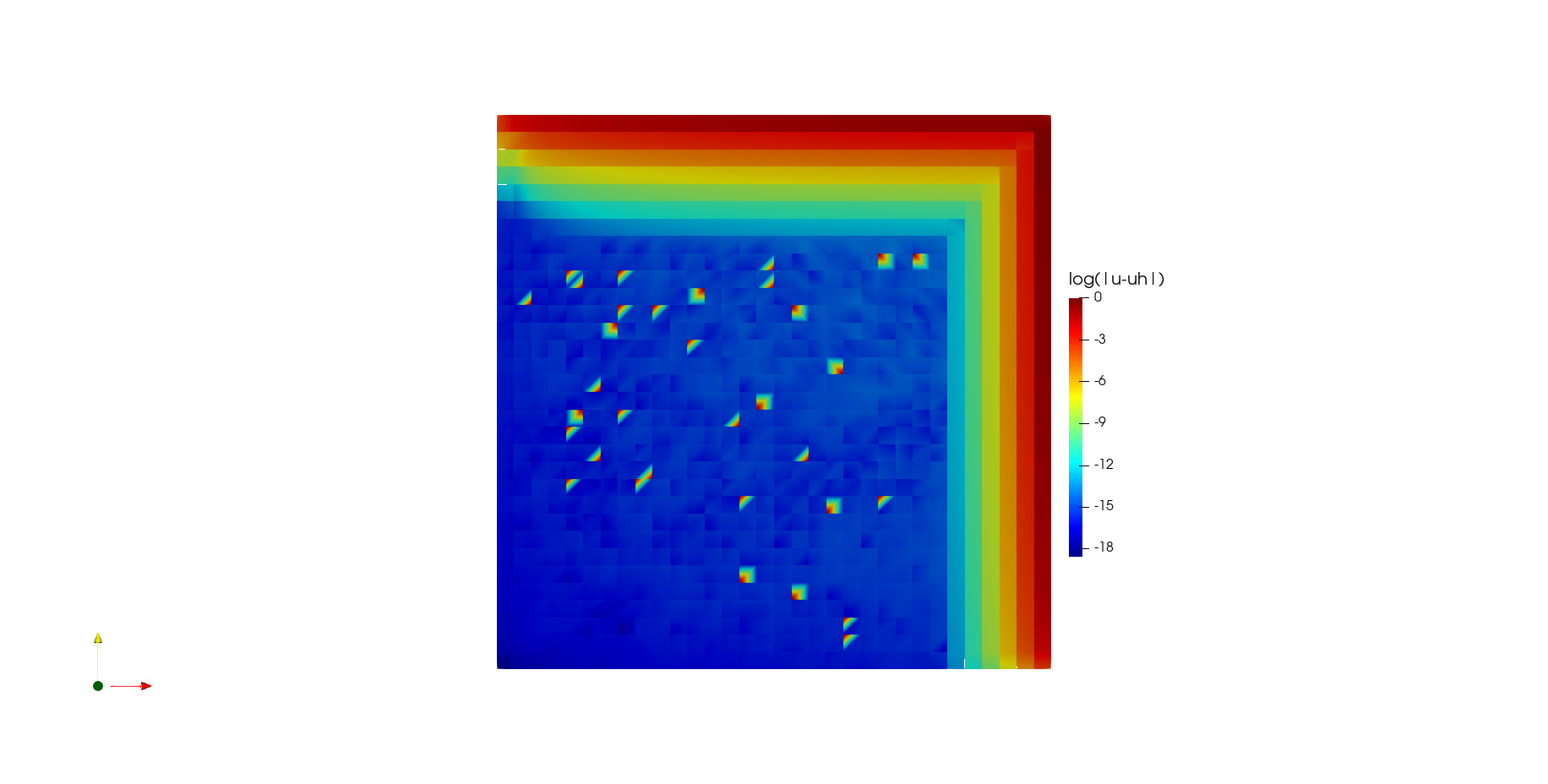}}
\caption{(a) Representation of the linear discrete solution $\du$ ($k=1$) using the Add-scheme for $\kappa=5e^{-3}$ and (b) its corresponding $\LL$-error spatial distribution. (c) Representation of the linear discrete solution $\du$ ($k=1$) using the SG-scheme for $\kappa=5e^{-3}$ and (d) its corresponding $\LL$-error spatial distribution.}
\label{Test-non-degenerate}
\end{figure}

\subsection{Fully degenerate problem}

In the second example, we analyze the behavior of the H-IP method in the context of hyperbolic problems. Let us notice that the discrete bilinear operator as defined in \eqref{HIP-method} is reduced to its advective-reactive part since $\bkappa$ is null throughout the whole domain $\Domain=[0,1]^2$. Thus, we set $\bbeta\eqbydef(2,1)$, $\gamma\eqbydef 1$ and $f\eqbydef 0$. The following exact solution is prescribed
\begin{equation}
    \label{heaviside}
    u(x,y) = H(-x+2y-1)\quad\textrm{for all}\quad(x,y)\in\Domain,
\end{equation}
where $H(\cdot)$ denotes the Heaviside function. Dirichlet boundary conditions are imposed (only) for the degenerate inflow part $\gammam$ -- \ie the left and bottom boundaries of $\Domain$. The exact solution displays a discontinuity along the characteristic direction $\bbeta$ due to the jump in imposed boundary conditions at $x=0$.  We also assume that the actual location of internal layers is unknown, and an adaptive mesh refinement strategy is investigated to capture them. Let us precisely state that resulting meshes are (generally) not aligned with the characteristic direction $\bbeta$. We analyze here the role of the upwind-parameter $\theta$ (see \eg Definition \ref{up_tau}) with respect to the accuracy of discrete solutions for different polynomial degrees $k\eqbydef\{1,2\}$. Since the exact solution \eqref{heaviside} is only piecewise constant, increasing polynomial degrees can only yield better discrete approximations near internal layers. Following \figref{Heaviside-solutions}, let us notice first that the $\theta$-upwind scheme effectively handles outflow boundary conditions for any selected value of $\theta $: this observation agrees with Remark \ref{degenerate-out-bound}. However, the increase in $\theta$ significantly deteriorates the discrete solution near the internal layer for all polynomial degrees since there is more erratic behavior and/or a more considerable artificial numerical diffusion: this observation agrees with Remark \ref{traditional-schemes}. Thus, we will favor the traditional upwinding scheme obtained by selecting $\theta=1$ in hyperbolic regions in all future experiments.\par
\begin{figure}[!ht]
\centering
\subfloat[$k=1$ and $\theta = 1$]{\includegraphics[scale=0.21,trim = 555 120 555 130, clip=true]{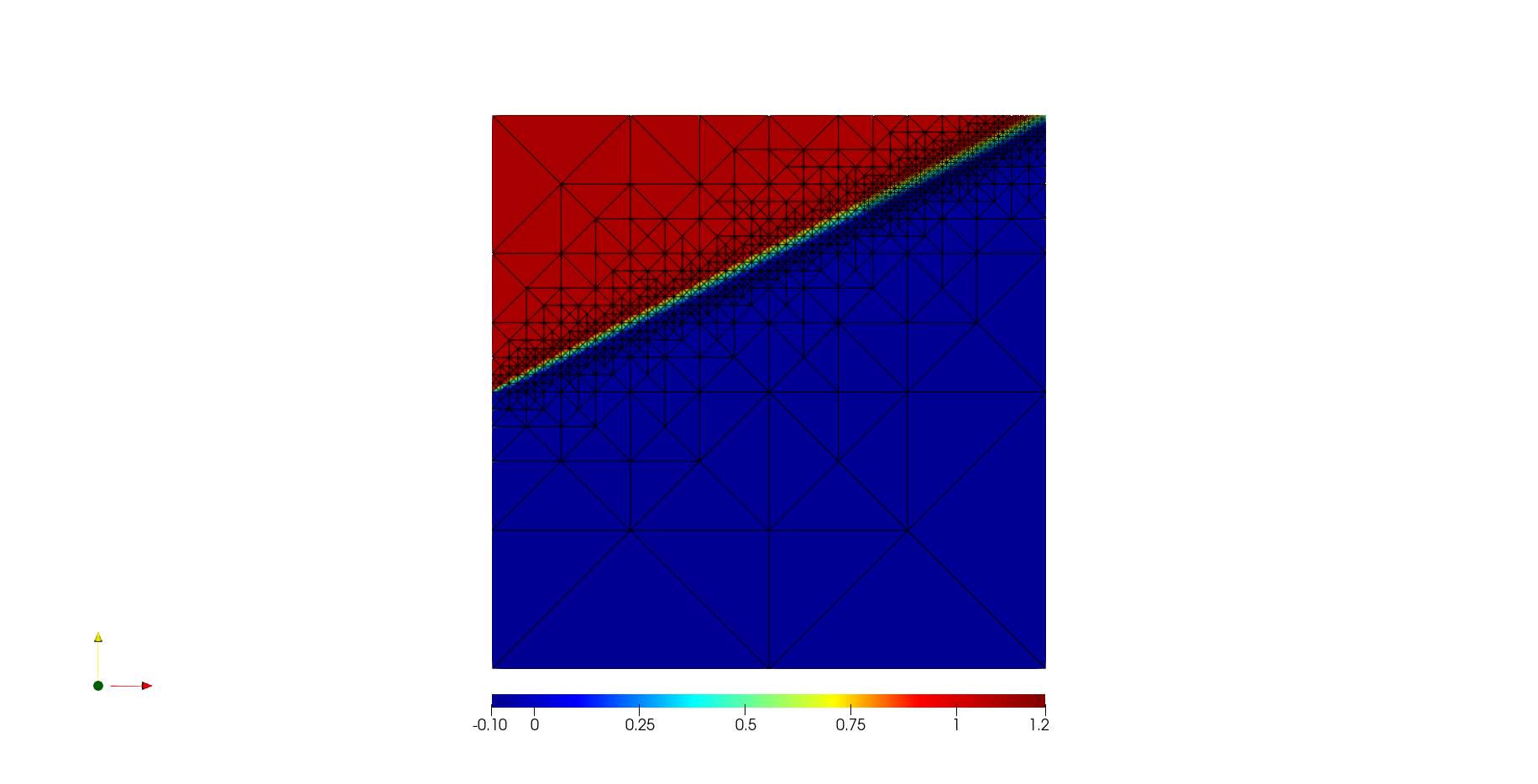}}  
\subfloat[$k=1$ and $\theta = 10$]{\includegraphics[scale=0.21,trim = 555 120 555 130, clip=true]{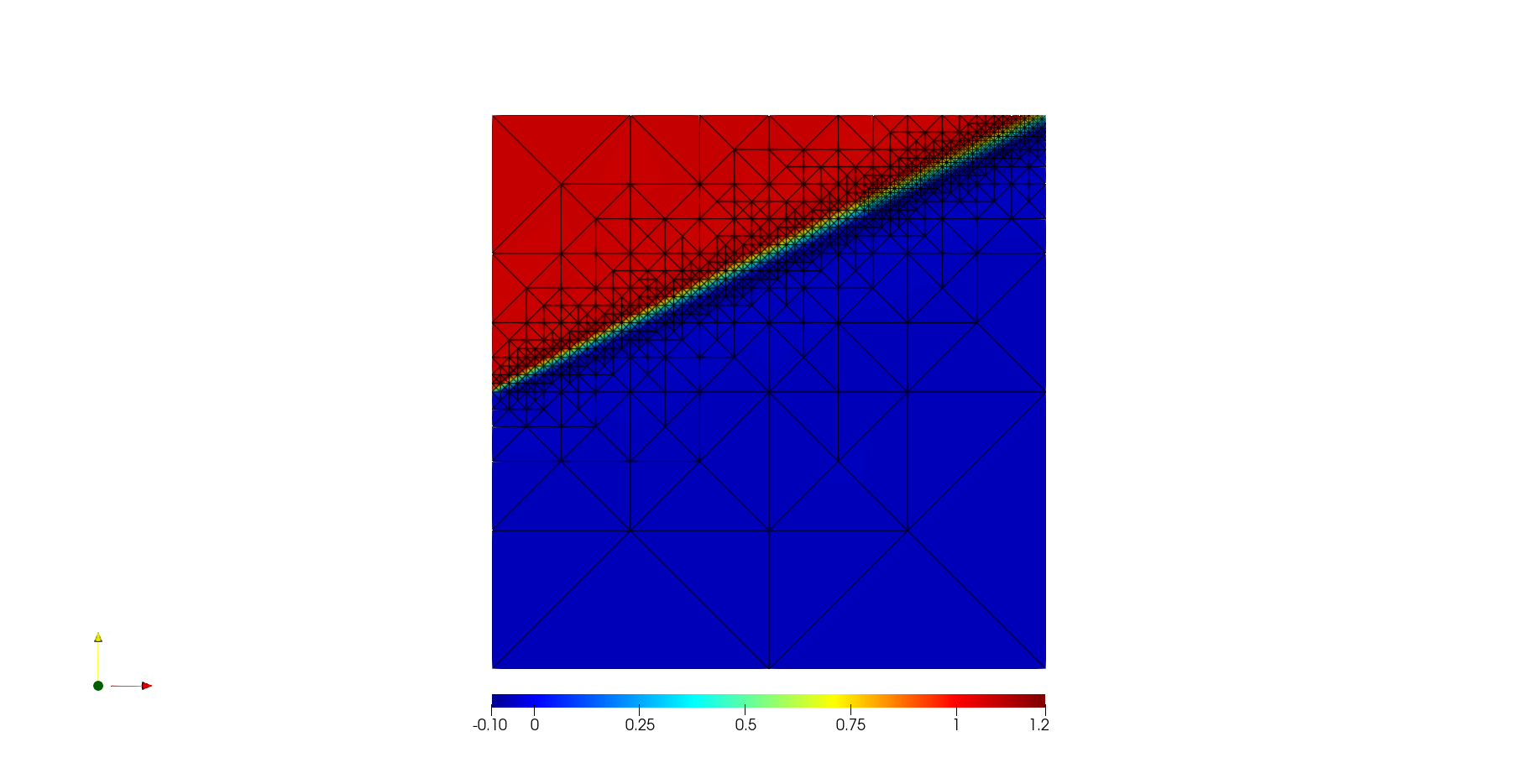}} 
\subfloat[$k=1$ and $\theta = 100$]{\includegraphics[scale=0.21,trim = 555 120 555 130, clip=true]{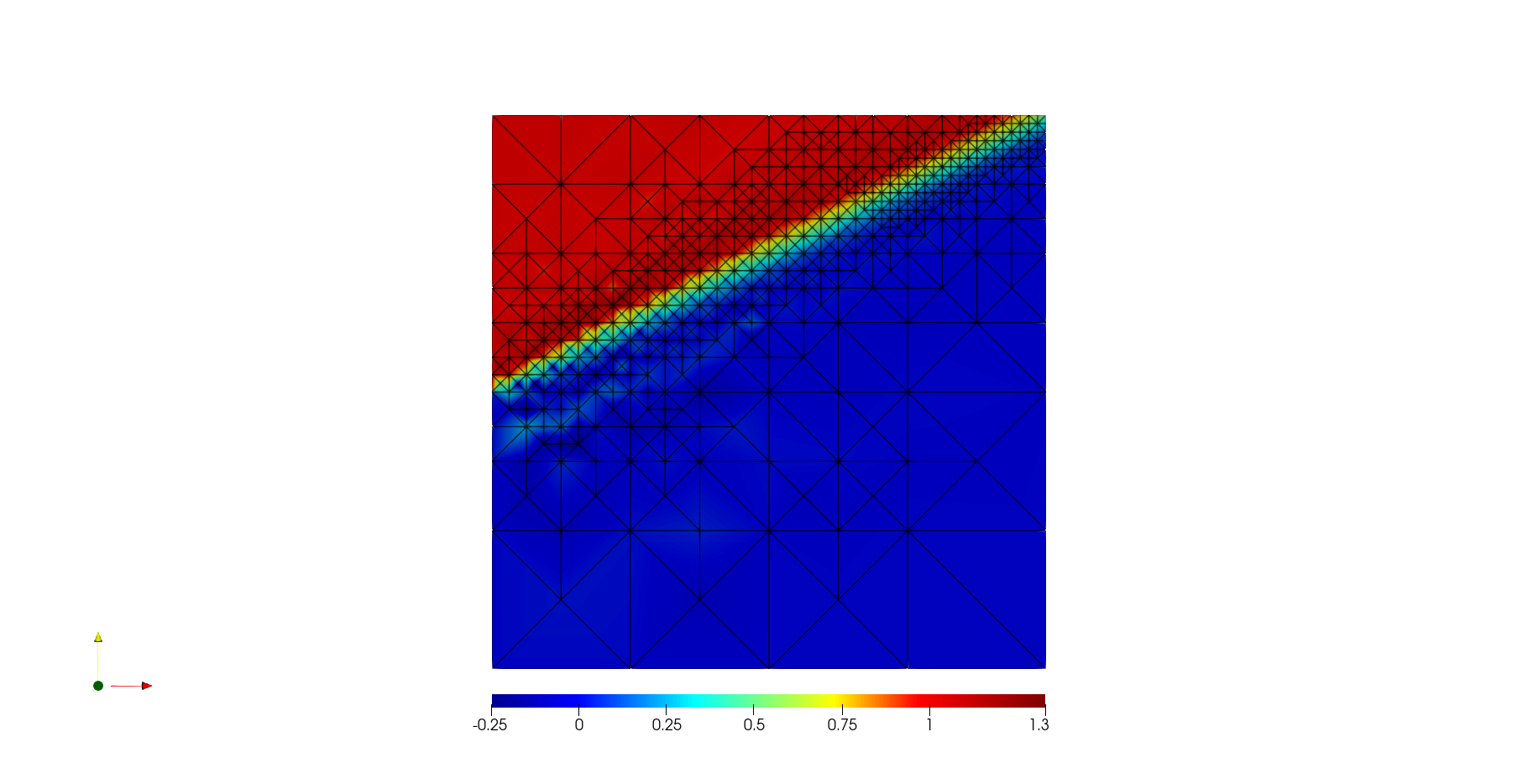}} \\
\subfloat[$k=2$ and $\theta = 1$]{\includegraphics[scale=0.21,trim = 555 120 555 130, clip=true]{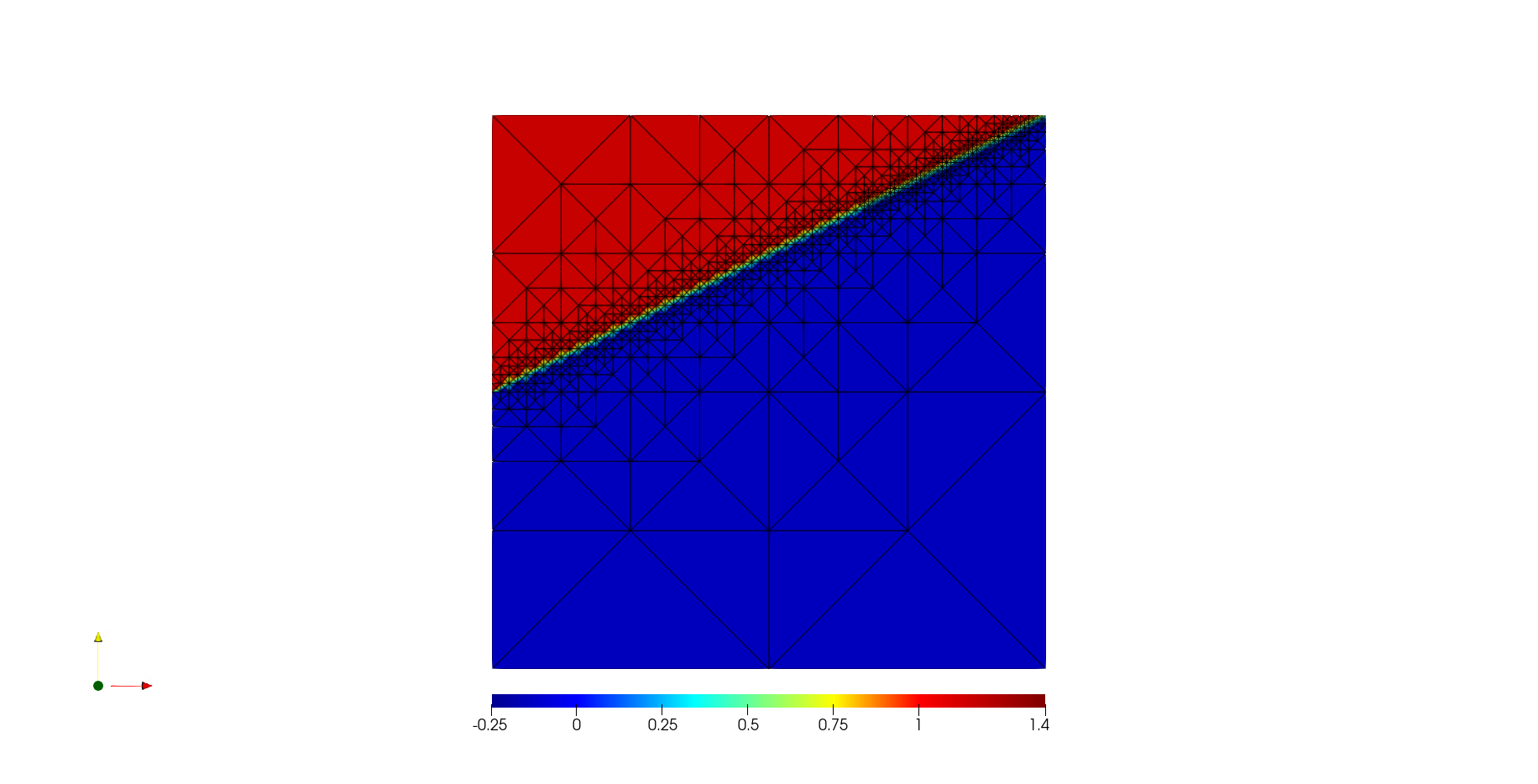}}  
\subfloat[$k=2$ and $\theta = 10$]{\includegraphics[scale=0.21,trim = 555 120 555 130, clip=true]{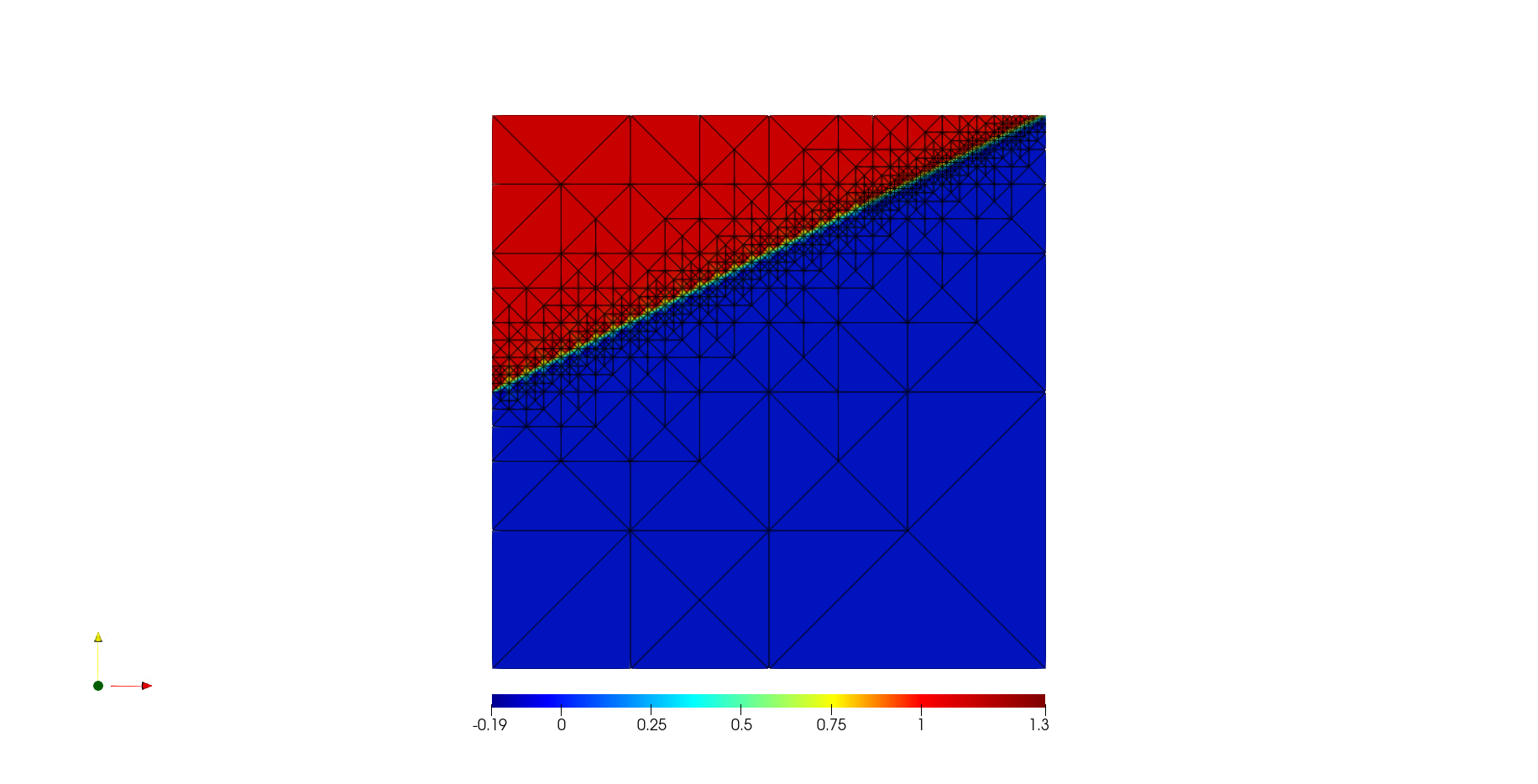}} 
\subfloat[$k=2$ and $\theta = 100$]{\includegraphics[scale=0.21,trim = 555 120 555 130, clip=true]{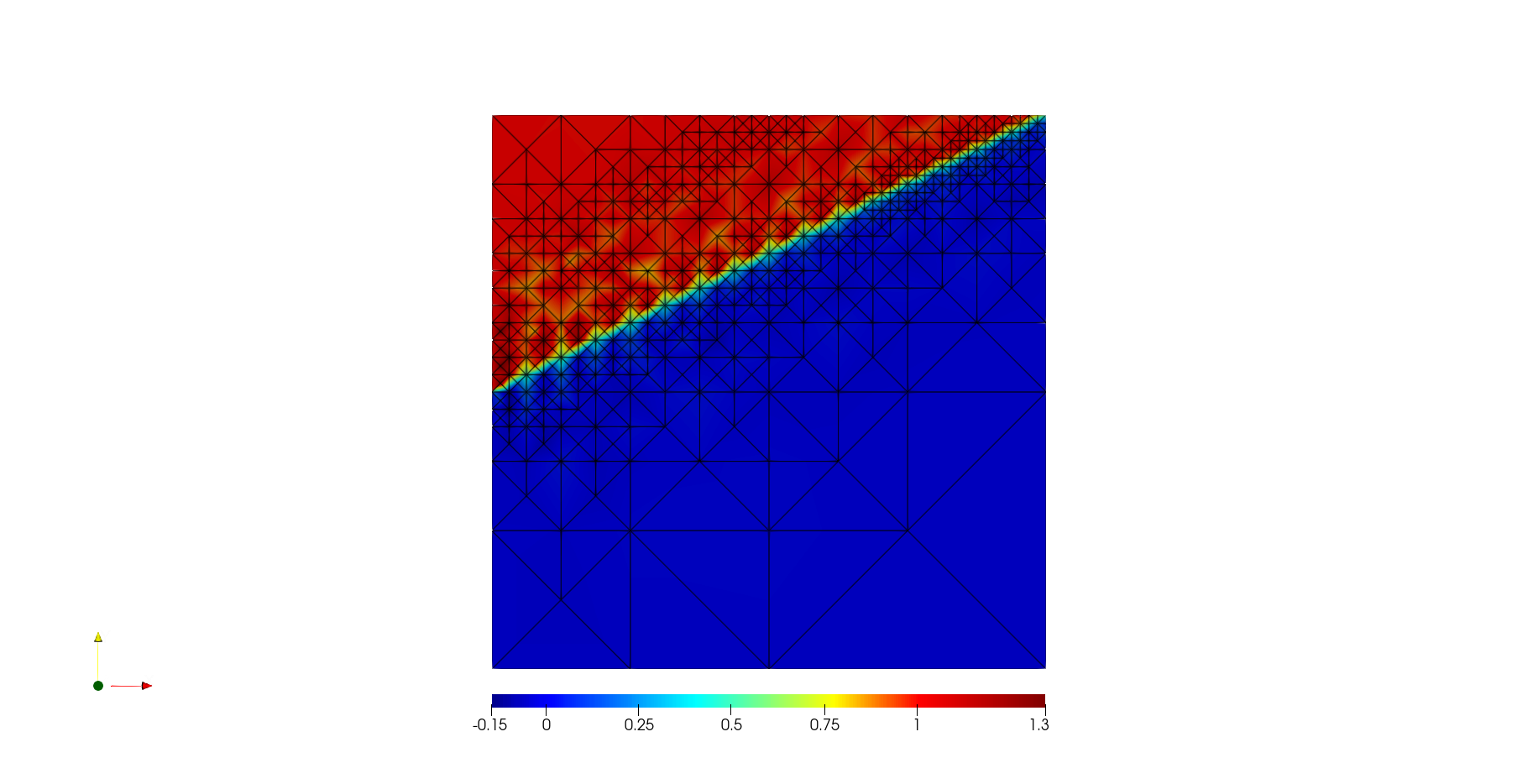}}
\caption{Representation of linear (a)-(b)-(c) and quadratic (d)-(e)-(f) discrete solutions obtained on adaptive meshes for different values of the upwind-parameter $\theta=\{1,10,100\}$, respectively.}
\label{Heaviside-solutions}
\end{figure}

\subsection{Locally degenerate problem}

In the last example, we validated the proposed H-IP formalism in the context of locally degenerate diffusion. To achieve this aim, we consider a slight modification to the test case proposed by Di Pietro \al in \citep{CDG2008} (see \eg \citep{dipietro:hal-01079342}). The domain is now taken to be $\Domain\eqbydef[-1,1]^2\backslash\{x^2+y^2<\frac{1}{4}\}$, which is divided into two disjoint subdomains $\Domain^{\textrm{ell}}$ and $\Domain^{\textrm{hyp}}$ corresponding to the elliptic and hyperbolic parts, respectively (see \eg \figref{Test-locally-degenerate}a). Denoting by $(r,\theta)$ the polar coordinates (with azimuth $\theta$ measured in the anticlockwise sense starting from the positive $x$-axis) and by $\vect{e_{\theta}}$ the (unit) azimuthal vector, the problem coefficients are 
\begin{equation}
    \kappa = \left\{ \begin{array}{ll}
        \pi & \textrm{if } 0 < \theta \leq \pi, \\
         0  & \textrm{if } \pi < \theta \leq 2 \pi ,
    \end{array} \right. \qquad
    \bbeta = \frac{\vect{e_{\theta}}}{r}, \qquad
    \gamma = 1e-6 , 
\end{equation}
and $\bkappa\eqbydef\kappa \vect{I_2}$. As illustrated in \figref{Test-locally-degenerate}b, the exact solution is given by,
\begin{equation}
    \label{exact-sol-locally-degenerate}
    u(r,\theta) = \left\{ \begin{array}{ll}
        (\theta - \pi)^2 & \textrm{if } 0 < \theta \leq \pi, \\
        3 \pi (\theta - \pi) & \textrm{if } \pi \leq \theta < 2 \pi .
    \end{array} \right.
\end{equation}
and is used to infer the forcing term $f$ and the imposed Dirichlet boundary datum. Here, we focus only on the H-SIP method using the Scharfetter--Gummel scheme in the elliptic region, and the H-IP method using the $\theta$-upwind stabilization strategy in the hyperbolic one. For clarity of our exposition, we set $\theta\eqbydef 1$ in both regions, but it is quite possible to fix distinctive values in each subdomain. Standard $h$- and $k$-refinement strategies are used to compute the discrete $\LL$-errors and estimated convergence rates (ECRs). A history of convergence is illustrated in \figref{convergence-degenerate}a for different polynomial degrees $k\eqbydef\{1,\ldots,5\}$. We observe that the convergence rate is optimal with order $k+1$ for all $k$.\par
\begin{figure}[!ht]
 \centering
 \subfloat[]{\begin{tikzpicture}[scale=1]
 \coordinate (A) at (0,0) ; 
 \coordinate (AI) at (3,0) ;
 \coordinate (B) at (6,0) ;
 \coordinate (BI) at (6,3) ;
 \coordinate (C) at (6,6) ;
 \coordinate (CI) at (3,6) ;
 \coordinate (D) at (0,6) ;
 \coordinate (DI) at (0,3) ;
 \draw[thick] (BI) -- (C) -- (CI) -- (D) -- (DI) -- (A) ;
 \draw[thick] (AI) -- (B) ;
 \draw [thick,densely dashdotted] (A) -- (AI);
 \draw [thick,densely dashdotted] (B) -- (BI);
 
 \draw[ thick] (4.5,3) arc(0:180:1.5);
 \draw[ thick,densely dashdotted] (1.5,3) arc(180:360:1.5);

 \coordinate (a) at (0,3) ; 
 \coordinate (b) at (6,3) ;
 \draw[ thick,densely dashdotted,carmine] (a) -- ++(1.5,0);
 \draw[ thick,carmine] (b) -- ++(-1.5,0);
 
 \draw[ thick, ->] (5.25,2.75) -- ++(0,0.5);
 \draw[ thick, ->] (0.75,3.25) -- ++(0,-0.5);
 
 \coordinate (A1) at (1.25,1.25) ; 
 \coordinate (B1) at (3.75,1.25) ;
 \coordinate (C1) at (3.75,3.75) ;
 \coordinate (D1) at (1.25,3.75) ;
  
 \draw (3.25,0.9) node {$\Domainh $};
 \draw (3.25,5.25) node {$\Domaine $};
 \draw[thick] (0.95,3.5) node {$\interfp$}; 
 \draw [thick](5.4,2.5) node {$\interfm$}; 
 
 \draw [thick](4.28,4.28) node {$\gammam$}; 
 \draw [thick](1.75,1.75) node {$\gammap$};
\end{tikzpicture}}
\subfloat[]{\includegraphics[scale=0.26,trim = 420 130 350 150, clip=true]{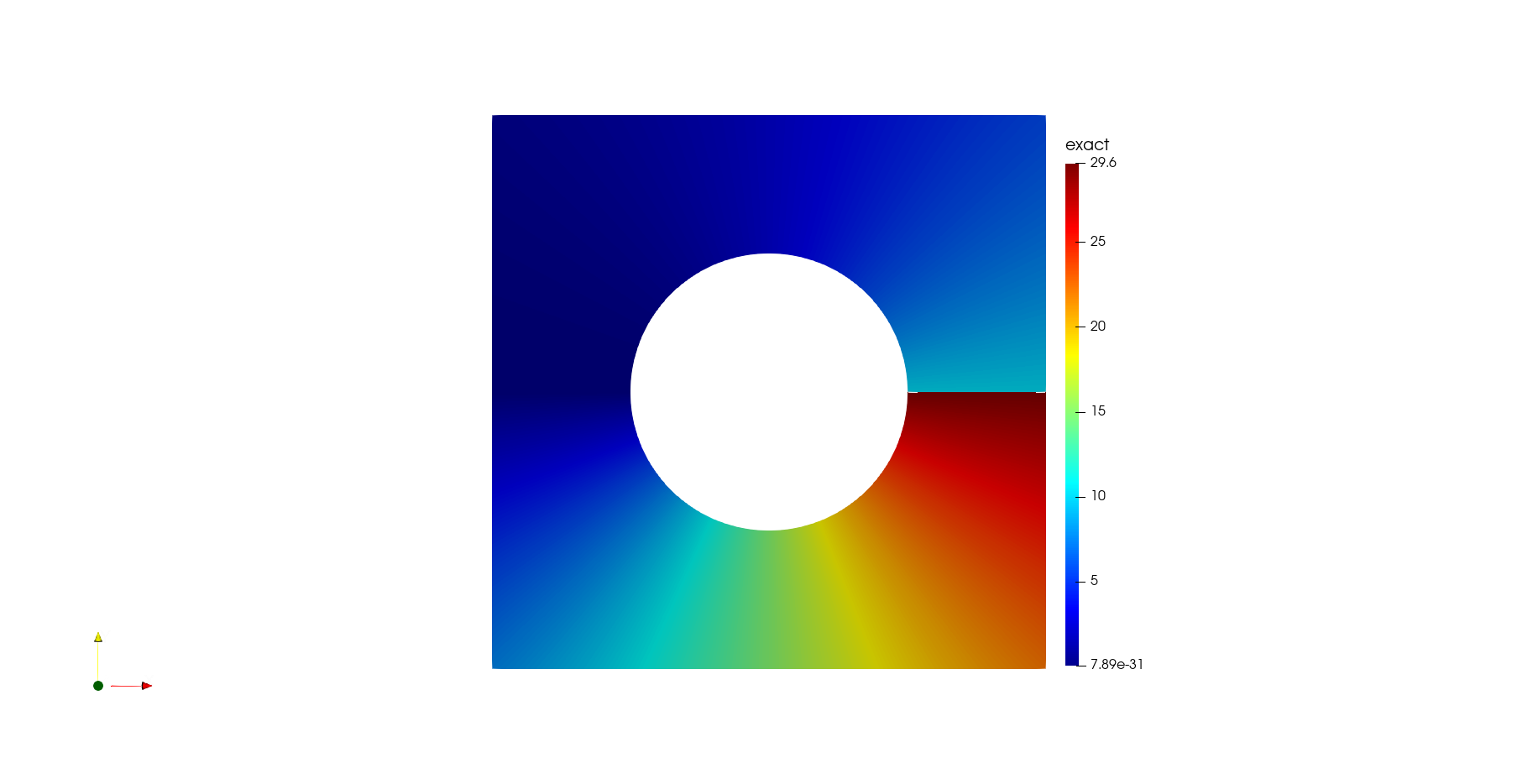}}
\caption{(a) Description of the locally degenerate test case: the continuous black line denotes the nondegenerate inflow boundary $\gammam$ where Dirichlet conditions are imposed, the dashed black line indicates the degenerate no-flow/outflow boundary $\gammap$, and the continuous red line represents the interior interface $\interfm$  where the exact solution is discontinuous. (b) Illustration of the exact solution \eqref{exact-sol-locally-degenerate}.}
\label{Test-locally-degenerate}
\end{figure}
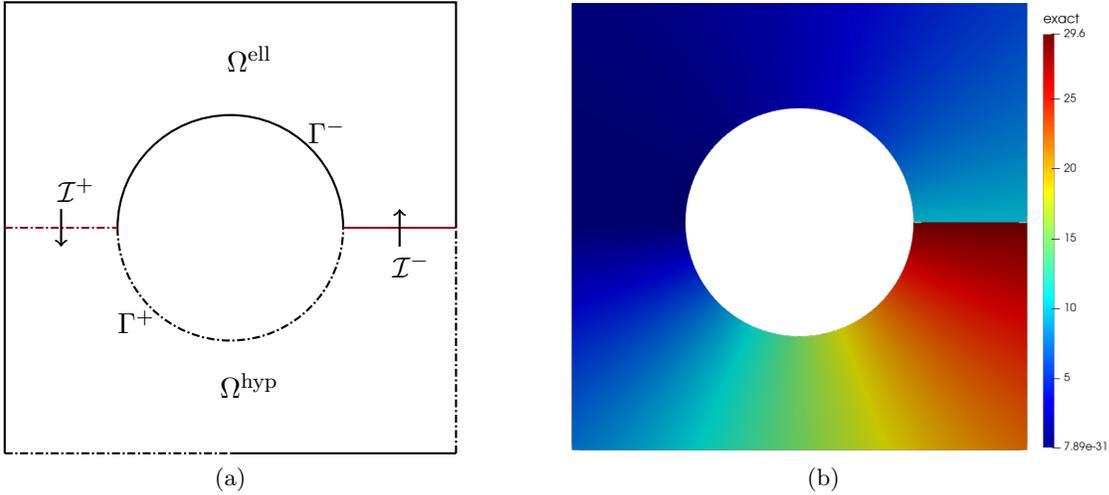
\begin{figure}[!ht]
\centering
\subfloat[]{\begin{tikzpicture}[scale=0.90]
	\begin{loglogaxis}[
		xlabel=$\log \left( 1/h \right)$,
		ylabel=$\log \left( \norm{u-\du}{0}{\Th{}} \right)$, legend cell align=left, legend pos = south east, font=\tiny]
		
	\addplot[mark=*] coordinates {
		(1/4,2.1e-01) (1/8,5.2e-02) (1/16,1.3e-02) (1/32,3.1e-03) (1/64,7.9e-04)
	};
	\addplot[mark=triangle] coordinates {
		(1/4,4.7e-02) (1/8,2.8e-03) (1/16,2.2e-04) (1/32,2.5e-05) (1/64,2.9e-06)
	};
	\addplot[mark=square] coordinates {
		(1/4,5.5e-03) (1/8,1.9e-04) (1/16,8.0e-06) (1/32,4.5e-07) (1/64,2.7e-08)
	};
	\addplot[mark=x] coordinates {
		(1/4,1.3e-03) (1/8,1.8e-05) (1/16,3.2e-07) (1/32,7.9e-09) (1/64,2.4e-10)
	};
	\addplot[mark=o] coordinates {
		(1/4,1.8e-04) (1/8,1.1e-06) (1/16,9.8e-09) (1/32,1.4e-10) 
	};
	\draw (axis cs:1/32,3.1e-03) |-  (axis cs:1/64,7.9e-04) node[near end,above] {$ $} node[near start,right] {$2.00$};
	\draw (axis cs:1/32,4.5e-07) |-  (axis cs:1/64,2.7e-08) node[near end,above] {$ $} node[near start,right] {$4.10$};
	\draw (axis cs:1/16,9.8e-09) |-  (axis cs:1/32,1.4e-10) node[near end,right] {$ $} node[near start,right] {$6.20$};
    \legend{$k=1$, $k=2$, $k=3$, $k=4$, $k=5$}

	\end{loglogaxis}
\end{tikzpicture}}
\subfloat[]{\includegraphics[scale=0.24,trim = 480 45 350 90, clip=true]{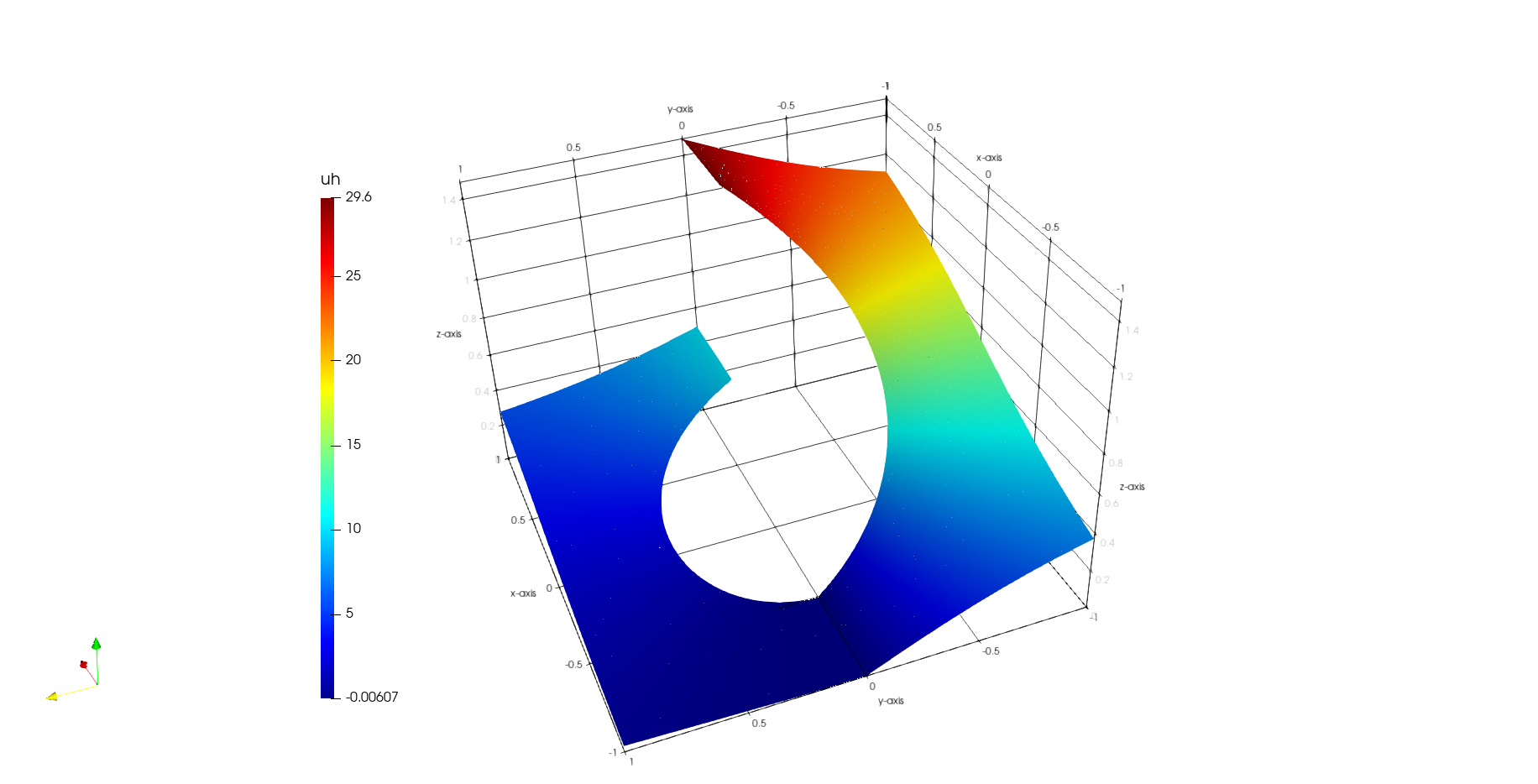}}
\caption{(a) History of convergence error in the $L_2$-norm for different polynomial degrees $k=\{1,\ldots,5\}$. (b) Representation of $\du$ using piecewise linear approximations ($k=1$) on a fine mesh ($h=1/64$).}
\label{convergence-degenerate}
\end{figure}
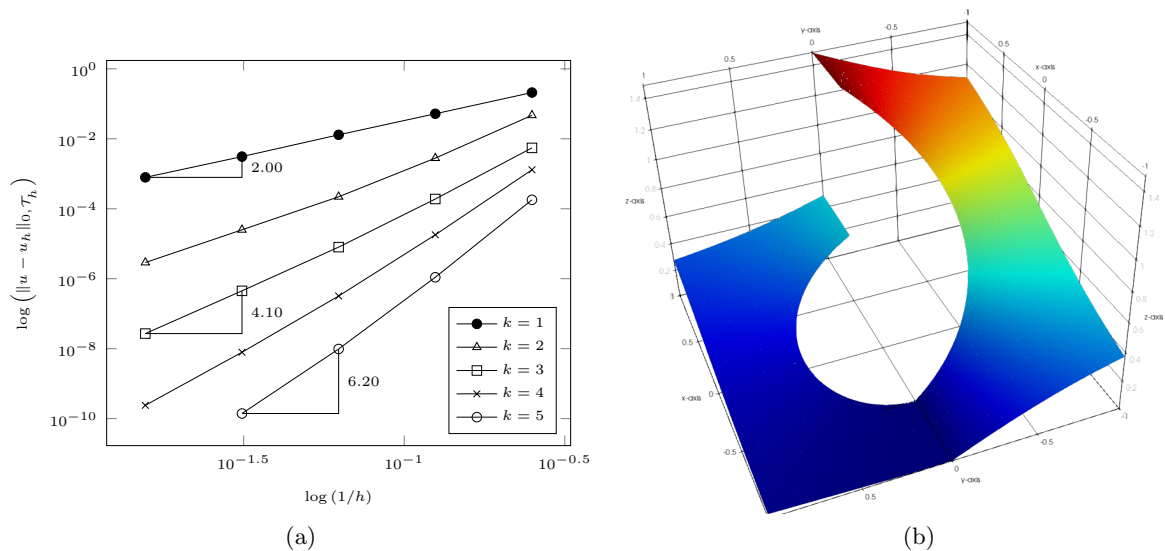

\section{Conclusion \& Perspectives}

We have derived a compact interior penalty HDG method for solving degenerate advection-diffusion-reaction problems. The proposed H-IP method can efficiently handle pure diffusive or advective regimes, along with intermediate regimes combining the above mechanisms for a wide range of P\'eclet numbers, including the delicate situation of local evanescent diffusion. An adaptive stabilization strategy is carried out, automatically accounting for the mathematical nature of \eqref{continuous-deg-problem} and the predominance of the diffusion or advection mechanisms. An upwinding-based scheme was favored for the hyperbolic region, and an inspired Scharfetter--Gummel scheme was preferred for the elliptic region. One undeniable advantage of this strategy is the possibility of tuning the amount of upwind. The stability analysis indicates that all considered variants are consistent and coercive, ensuring the well-posedness of the discrete problem in all regimes. The flexibility and accuracy of the proposed method are confirmed by numerical lines of evidence.

\section*{Acknowledgements}
\label{Ack}
By convention, the names of the authors are listed in alphabetical order. The third author is grateful to Sander Rhebergen for his invitation to the Department of Applied Mathematics at the University of Waterloo (UW) in April 2019. The present work is undoubtedly the fruit of our first discussions and investigations during my visit to the UW.


\bibliography{HIP_Deg}

\end{document}